%% file: mfld_arxiv.tex
\documentclass{article}
\usepackage{a4wide,eucal,enumerate,mathrsfs,xspace}
\usepackage{amsmath,amssymb,epsfig,bbm}
\usepackage[toc,page]{appendix}
\usepackage[usenames]{color}
\usepackage{tikz}
\usepackage{cite}
\usepackage{ifthen}
\usepackage[a4paper,top=2cm,bottom=2cm,left=2cm,right=2cm,bindingoffset=5mm]{geometry}

\input macro2014
\usetikzlibrary{matrix}

\numberwithin{equation}{section}

\newtheorem{theorem}{Theorem}[section]

\newtheorem{lemma}[theorem]{Lemma}
\newtheorem{proposition}[theorem]{Proposition}
\newtheorem{definition}[theorem]{Definition}

\newtheorem{remark}[theorem]{Remark}
\newtheorem{hypothesis}[theorem]{Hypothesis}

\newenvironment{system} 
{\left\lbrace\begin{array}{@{}l@{}}}
{\end{array}\right.}

\newcommand{\mres}{\mathbin{\vrule height 1.6ex depth 0pt width
0.13ex\vrule height 0.13ex depth 0pt width 1.3ex}}

\newcommand{\dive}{\operatorname{div}}
\newcommand{\ent}{\operatorname{H}}

\newcommand{\law}{\operatorname{Law}}

\title{Large deviations for interacting particle systems: joint mean-field and small-noise limit}
\author{ Carlo Orrieri \thanks{Dipartimento di Matematica ''F. Casorati'', Universit\`a degli studi di Pavia. Via Adolfo Ferrata 5, 27100 Pavia, Italy. \texttt{carlo.orrieri01@ateneopv.it}}}

\begin{document}

\maketitle

\begin{abstract}
We consider a system of stochastic interacting particles in $\R^d$ and
we describe large deviations asymptotics in a joint mean-field and small-noise limit.
Precisely, a large deviations principle (LDP) is established for the empirical measure and the stochastic current, as the number of particles tends to infinity and the noise vanishes, simultaneously.
We give a direct proof of the LDP using tilting and subsequently exploiting the link between entropy and large deviations.
To this aim we employ consistency of suitable deterministic control problems associated to the stochastic dynamics.  
\end{abstract}

\tableofcontents

\section{Introduction}
In statistical mechanics, macroscopic properties of a physical system are usually derived from a probabilistic description of complicated interactions at a microscopic level.
Generally, the macroscopic behaviour is provided by means of a deterministic partial differential equation, also known as hydrodynamical equation.  
At a microscopic scale instead, the dynamic can be described via a stochastic interacting particles model whose choice is then fundamental to get a rigorous derivation of the above macroscopic equation. 
A step further in the study of (out of equilibrium) systems consists in understanding whether it is possible, and how large is the probability, to observe a different macroscopic behaviour from the one predicted by hydrodynamics.
To answer this question it is quite natural to look for a large deviations principle (LDP in short), for which some fluctuations around the equilibria of the quantities involved are also captured.

\smallskip

Within this framework, a  laboratory but rich enough example to investigate is the one proposed by McKean in the context of propagation of chaos, see e.g. \cite{sznitman1991topics}.
Given a bunch of particles randomly moving in the whole space $\R^d$, we prescribe their evolution with a system of It\^o-type SDEs with independent Brownian noises.    
The interaction between the (exchangeable) particles is required to be of mean field type, i.e. each particle depends on the current empirical distribution of the system, and the coefficients of all the equations have the same functional form.
Here, the relevant physical quantity to deal with is the particles density, and it has been proved in several situations that the associated empirical measure gives rise, after a proper rescaling,  to a macroscopic density solving a Vlasov-type equation.
The mean field character of the interaction is fundamental in this procedure: it guarantees that the contribution of any given particle to the empirical distribution is small when a sufficiently large number of particles is considered.  
Also, from a different perspective, the limit PDE can be thought as a model simplification of the $N$-particles system and can be used to investigate properties of the microscopic system as the number of particles is very large.

\smallskip

The present paper is an attempt to clarify relations among the various descriptions of mean field systems in $\R^d$, focusing on the micro/macro and deterministic/stochastic dualities.  
A rough picture of the problem is given in the following diagram:

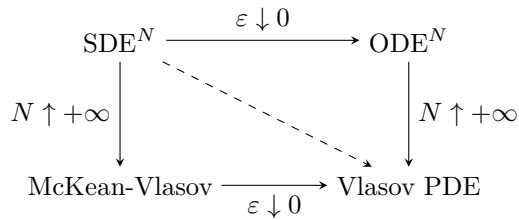
\begin{center}
\begin{tikzpicture}
  \matrix (m) [matrix of math nodes,row sep=4em,column sep=4em,minimum width=2em]
  {
     \text{SDE}^N & \text{ODE}^N \\
     \text{McKean-Vlasov} & \text{Vlasov PDE} \\};
  \path[-stealth]
    (m-1-1) edge node [left] {$N \uparrow +\infty$} (m-2-1)
            edge node [above] {$\varepsilon \downarrow 0$} (m-1-2)
            edge [dashed] (m-2-2)
    (m-2-1.east|-m-2-2) edge node [below] {$\varepsilon \downarrow 0$}
            node [above] {} (m-2-2)
    (m-1-2) edge node [right] {$N \uparrow + \infty$} (m-2-2) ;
\end{tikzpicture}
\end{center}

In one corner, from a microscopic-stochastic point of view, the system is modelled through $N$ stochastic equations with interaction (SDE$^N$), as briefly outlined above.
A counterpart of this description is given on the opposite side in a microscopic-deterministic fashion, where  $N$ deterministic differential equations (ODE$^N$)  govern the dynamic.
The relation between the two pictures is a well studied topic in the context of Friedlin-Wentcell theory of random perturbation of dynamical systems and it is represented in the above diagram by the arrow with vanishing noise  ($\varepsilon \downarrow 0$).

As the number of particles increases ($N \uparrow +\infty$), in the lower left corner  we deal with a macroscopic limit process (McKean-Vlasov) also referred to as  nonlinear diffusion.
Here we have to take into account both the limit behaviour of a typical particle and the limit of the empirical distribution. In fact, the nonlinear character of the diffusion originates from the fact that the dynamic of a typical particle depends on the particle distribution itself.

For what concerns the macroscopic characterisation on the right hand side (Vlasov PDE) we have at our disposal at least two different approaches.
On one hand, starting from the microscopic-deterministic model  (ODE$^N$) and sending $N \uparrow + \infty$, we obtain a continuity equation whose velocity field depends on the solution itself. This is in line with the usual mean-filed limit for finite dimensional interacting systems to which a large literature is devoted. 
On the other hand, a Vlasov-type PDE can also be obtained by a vanishing viscosity procedure ($\varepsilon \downarrow 0$) starting from the nonlinear diffusion for the law of the McKean-Vlasov process.
  
\smallskip

It seems natural to wonder whether the above diagram possesses some form of commutativity: Is it possible to freely interchange the two limit operations $N \uparrow +\infty$ and $\varepsilon \downarrow 0$? and to what extent?
With this question in mind, in the present manuscript we   go a step further in the analysis by considering large deviations asymptotics for the empirical measure and the associated stochastic current, trying to capture also fluctuations around the equilibria as $N \uparrow +\infty$ and $\varepsilon \downarrow 0$.

Partial answer to this question are present in the literature.
For what concerns the limit $N \uparrow + \infty$, the main reference for large deviations of stochastic mean field particle systems we refer to is \cite{dawson1987large}, see also \cite{oelschlager1984martingale, kipnis1990large}. In \cite{dawson1987large} the authors deal with uniformly nondegenerate diffusions, interacting through the drift term, and they derive a LDP for the empirical measure via a careful discretization procedure. 
A subsequent result in this direction has been obtained in \cite{budhiraja2012large} where the authors adopted a weak convergence approach combined with a variational representation result for moments of nonegative functionals of Brownaina motion \cite{boue1998variational}. This strategy actually bypasses the above mentioned discretization procedure as well as exponential probability estimates, and could cover some models with interaction in the diffusion.
Many other generalizations/directions have been explored in the literature, let us refer to e.g. \cite{dawson1994multilevel} for multilevel LDPs, \cite{delmoral1998large} for discrete-time systems \cite{daipra1996mckean} for what concerns random environment, \cite{leonard1995large} for Jump processes, \cite{deuschel2018enhanced} in the rough path setting and \cite{lacker2017limit} for the application to the theory of control.
Also the deterministic counterpart of the mean field theory ($N \uparrow +\infty$) is by now an active area of research collecting motivations ranging from physics to biology, from social sciences to control theory. In the last    
decade there has been a significant effort in providing rigorous derivation of the PDE models starting from finite dimensional systems.
For a general result in this direction we refer to \cite{canizo2011well}, where a well-posedness theory for some kinetic models is taken into account. See also \cite{fornasier2018mean} for further references and an application to optimal control problems.  
The convergence $\varepsilon \downarrow 0$ at the level of the particles system fits into the framework of Friedlin-Wentcell theory \cite{freidlin2012random}.  
Whereas for what concerns the nonlinear diffusion, 
a LDP for McKean-Vlasov equations in the small noise limit has been firstly established in \cite{herrmann2008large} and then generalized in many directions, see e.g. the recent \cite{reis2017freidlin} and the references therein.
Recall that, at a purely PDE level, the limit $\varepsilon \downarrow 0$ coincides with a vanishing viscosity limit for the nonlinear diffusion towards the solution to the Vlasov PDE. 

Finally, in \cite{herrmann2016mean} the authors addressed the problem of interchanging mean-field limit with the small noise-one. What they proved  is that the rate functional associated with the first particle in a mean field system actually converges to the rate functional of the hydrodynamical equation as $N$ becomes large.  

 \smallskip

In this paper we further study the combination of mean-field limit and small-noise limit by establishing  a LDP for the empirical measure and stochastic current as $\varepsilon \downarrow 0, N \uparrow +\infty$ simultaneously. 
A general motivation for studying LDPs for the pair measure-current comes from non-equilibrium statistical physics, in which the current is an important observable of the system. 
No less importantly, within this framework an explicit formula for the rate functional is often feasible  and the corresponding  LDP formulation for the empirical measure can be obtained by contraction.


More specifically, in the present setting  we consider $N$ particles $(x_1, \ldots, x_N)$ in $\R^d$ solving the system of SDEs:
\begin{equation}
  \label{eq:state_intro}
  \d x_i^{N,\varepsilon}(t)=\fF(x_i^{N,\varepsilon}(t),\xx^{N,\varepsilon}(t)) \d t + \sqrt{\varepsilon}\d W_i(t) ,\quad i=1,\cdots, N,\qquad
  \xx^N(0)=\xx^N_0.
\end{equation}
where the map $\fF:\R^d\times (\R^d)^N\to \R^d$ models the mean-field interaction and $(W_i(t))_{t \in [0,T]}$ are $d$-dimensional independent Brownian motions.
We associate to the system the empirical measure
\begin{equation}\label{empirical:measure_intro}
\mu^{N,\varepsilon}_t := \frac{1}{N}\sum_{i=1}^N \delta_{x^{N,\varepsilon}_i(t)}
\end{equation}
and for every $\eta \in C^\infty_c([0,T] \times \R^d; \R^d)$ we define the stochastic current as
\begin{equation}\label{stoch:current_intro}
J^{N,\varepsilon}(\eta) := \frac{1}{N}\sum_{i=1}^N \int_0^T \eta(t,x_i^{N,\varepsilon}) \circ \d x^{N,\varepsilon}_i(t).
\end{equation}
If we denote by $\P^{N,\varepsilon}$ the law of the solution to \eqref{eq:state_intro} and by $\cX$ the space 
\[\cX: = C([0,T];\PP_1(\R^d)) \times H^{-s_1}\left((0,T); H^{-s_2}(\R^d;\R^d)\right),\]
we aim at showing that the probability measures  
$\P^{N,\varepsilon} \circ (\mu^N, J^{N,\varepsilon})^{-1} \in \PP(\cX)$ satisfy a LDP in $\cX$ with speed $\varepsilon /N$ and (good) rate functional given in a variational form by 
\begin{equation}\label{rate:intro}
\begin{split}
I(\mu,J)  = \sup_{\eta \in C^{\infty}_c(Q;\R^d)} \Big\lbrace  &J(\eta)  - \int_0^T \langle \mu_t, \eta(t, \cdot) \cdot \fF(\cdot, \mu_t) \rangle \d t \\
& - \frac{1}{2}\int_0^T \langle \mu_t, \left| \eta(t,\cdot)  \right|^2 \rangle \d t \quad \Big| \;\partial_t \mu_t  + \nabla \cdot J = 0 \Big\rbrace.
\end{split} 
 \end{equation}
This is to say that for any Borel set $\sfB \subset \cX$
 \begin{equation}
 \begin{split}
- \inf_{(\mu, J) \in \mathring{\sfB}} I(\mu, J) \leq \liminf_{\substack{\varepsilon \to 0\\ N \to +\infty}} \frac{\varepsilon}{N} \log \P^{N,\varepsilon} &((\mu^{N}, J^{N,\varepsilon}) \in \mathring{\sfB}) \\
&\leq  \limsup_{\substack{\varepsilon \to 0\\ N \to +\infty}} \frac{\varepsilon}{N} \log \P^{N,\varepsilon} ((\mu^{N}, J^{N,\varepsilon}) \in \bar\sfB) \leq - \inf_{(\mu, J) \in \bar\sfB} I(\mu, J),
\end{split}
\end{equation}
independently of the order of the limits in $\varepsilon$ and $N$ (see Theorem \ref{t:LDP} for a precise statement with minimal assumptions).
The proof of the above result is exhibit in a direct way, finding a microscopic perturbation of the system, macroscopically non trivial, from which it is possible to get the correct form of the large deviation functional.
This is more evident writing the rate functional in an explicit fashion.
In fact, denoting $\tilde \mu:= \int_0^T \delta_t \otimes \mu_t \, \d t$,  when $I (\mu,J)<+\infty$ there exists a vector field $\hh \in L^2(Q,\tilde \mu;\R^d)$ for which $\d J = \hh  \d \tilde\mu$  and
\begin{equation}
I(\mu,\hh) = \frac{1}{2} \int_0^T \langle \mu_t, |\hh(t,\cdot)|^2 \rangle \d t,
\end{equation}
for all the pairs $(\mu,\hh)$ satisfying  in a distributional sense
\begin{equation}\label{eq:cont:intro}
\partial_t \mu_t + \dive \left( (\fF(\cdot, \mu_t) + \hh)\mu_t \right)= 0.
\end{equation}
Within this setting, the formulation of the LDP for the pair measure-current is fundamental to get the explicit form of the rate functional above. 
Also notice that in the limit $\varepsilon \downarrow 0$, given a  measure $\mu$ there exist more than one current $J$ for which the continuity equation is satisfied and the application of a contraction principle is not trivial.  
In the Appendix (see Theorem \ref{ex_uniq_vlasov}) we finally give some sufficient condition on the velocity field $\vv:= \fF + \hh$ and on the initial datum $\mu_0$ for which wellposedness of the Vlasov PDE \eqref{eq:cont:intro} is guaranteed. 

The proof of the large deviations upper bound is constructed by a specific tilting of the measures, providing the right estimates on compact sets. To get the required expression for closed sets a careful exponential tightness argument for both empirical measure and stochastic current is needed.
The proof of the lower bound is more delicate.
Firstly, in  Theorem \ref{t.lb_entropia} we exploit the relation between large deviations and $\Gamma$-convergence developed in \cite{mariani2012gamma}: this theorem is fundamental as it translates the lower bound estimate in a $\Gamma$-$\limsup$ inequality.
Then, we take advantage of the result obtained in \cite{fornasier2018mean} to construct a suitable recovery sequence.
More precisely, in \cite{fornasier2018mean} the authors studied the interplay between finite and infinite dimensional control problems for multi-agents systems and they obtained a $\Gamma$-convergence result (as the number of agents goes to infinity) under weak assumptions on the interaction kernel as well as on the cost functional.
This is crucial because the rate function $I(\mu, \hh)$ corresponds to a particular choice of the cost treated in \cite{fornasier2018mean} and it turns out that the recovery sequence actually provides a good perturbation of the system of SDEs for which the associated entropy remains controlled, see Theorem \ref{t:recovery_det} and Theorem \ref{t.LDP_lb_1}.

\smallskip

The paper is organised as follows. 
Some preliminary material concerning measure theory and topological issues, basic large deviations definitions, property of stochastic currents and solutions to the continuity equations are collected in Section 2.
Section 3 is devoted to the setting of the problem, hypotheses and main results.
The large deviations upper bound is discussed in Section 4 along with exponential tightness estimated and goodness of the rate functional.
In Section 5 the strategy to get the large deviations lower bound is presented, and the proofs of the main theorems are exhibited.
Finally, some sufficient conditions to have wellposedness of the Vlasov PDE are presented in the Appendix.

\section{Notation and preliminaries}

The following notation will be used throughout the paper. 

We fix $(\Omega, \cF, \P)$ a probability space endowed with a filtration $(\cF_t)_{t \in [0,T]}$ satisfying the usual conditions as well as a family $\lbrace W_i, i \in \N \rbrace$ of independent $d$-dimensional Brownian motions. 
Given a topological space $X$, we write $\PP(X)$ for the space of Borel probability measures on $X$. We endow $\PP(X)$ with the topology of weak (equivalently, narrow) convergence, in duality with bounded continuous functions $C_b(X)$.
In the special case $X = \R^d$, we will also use the notation $\PP_p(\R^d)$ referring to probability measures with finite $p$th-order moment:
\[ \cP \in \PP_p(\R^d) \Longleftrightarrow \cP \in \PP(\R^d) \text{ and } \int_{\R^d} |x|^p \d \cP(x) < +\infty, \qquad p\geq 1. \]  
The space of Borel and vector-valued Borel measures is denoted by $\MM(\R^d)$, $\MM(\R^d;\R^d)$, respectively.  
Given $X,Y$ two topological spaces, $\cP \in \PP(X)$ and a map $f : X \to Y$, we alternately use the notation $\cP \circ f^{-1}$ or $f_\sharp \cP$ to denote the push forward, or the image law, of the probability measure $\cP$ under the map $f$. 
We furthermore refer to the compact-open topology on $C(X,Y)$ as to the topology whose corresponding subbase is given by the sets $W(K,U) = \lbrace f \in C(X,Y): f(K) \subset U \rbrace$ as $K,U$ range over all compact subsets of $X$ and opens subsets of $Y$, respectively.
 
We indicate by $Q$ the cylinder $Q := (0,T) \times \R^d$ and we write $a \cdot b$ for the scalar product in $\R^d$, $a,b \in \R^d$.
$C_c^\infty(Q)$ stands for the set of smooth compactly supported functions in $Q$ and the notation
$C_c^{1,2}(Q)$ will be used for the set of compactly supported functions in $Q$ which are $C^1$ in time and $C^2$ in space. 
Given $D$ a smooth domain in $\R^d$, the fractional Sobolev spaces $W^{s,2}(D)$, $s \in \R$, will be shorthand $H^s(D)$.
If $\mu \in \DD'(X)$ is a distribution (respectively, $\mu \in \PP(X)$), we denote by $\langle \mu, \phi \rangle = \mu(\phi)$ the duality pair with a smooth function $\phi \in C_c^\infty(X)$ (respectively $\phi \in C_b(X)$). 
Lastly, when we write $a \lesssim b$ we mean that there exists a positive constant $C$ for which $a \leq Cb$. 

\medskip

Let us now recall a version of the Gronwall lemma which will be used in the sequel: 
if $a \in L^1(t_0,t_1)$, $c \geq 0$ and $u \in C([t_0,t_1];\R)$ satisfy
\begin{equation}
u(t) \leq a(t) + c \int_{t_0}^tu(s)\d s, \qquad \forall \, t  \in [t_0,t_1],
\end{equation} 
then the following estimate holds
\begin{equation}\label{gronwall}
u(t) \leq a(t) + c \int_{t_0}^t a(s) e^{c(t-s)} \d s, \qquad \forall \; t \in [t_0,t_1]
\end{equation}

\subsection{Some useful results in measure theory}
\label{sec.meas-theory}
A completely regular space $E$ is a topological space such that for every closed set $C \subset E$ and every point $x \in E \setminus C$ there exists a continuous function $f : E \to [0,1]$ for which $f(x) = 1$ and $f(y) = 0$, for every $y \in C$. Roughly speaking, it is possible to separate $x$ and $C$ with a continuous function. 
A completely regular space which satisfies the Hausdorff condition is called Tychonoff space.  
Notice that, given $X$ a normed space, the weak* topology of $X'$ is Tychonoff.
Moreover, if $X$ is separable, bounded sets in $X'$ are metrizable.

We say that a map $f: E \to F$ between two topological spaces is Borel measurable if $f^{-1}(A)$ is a Borel set, for any open set $A$. 
We denote by $\cM(E)$ the collection of real-valued Borel measurable maps. 
If $E$ is a topological space and $\cF \subset \cM(E)$, we say that $\cF$ \textit{separates points} of $E$ if for $x \neq y \in E$ there exists $h \in \cM$ such that $h(x) \neq h(y)$.
Again, we say that a set $\cG \subset \cM(E)$ is \textit{separating} for $E$ if given $\cP, \cQ \in \PP(E)$ with the property
\begin{equation}
\int_E h(x) \d \cP(x) = \int_E h(x) \d\cQ(x) \qquad \forall \, h \in \cG,
\end{equation}
then it follows that $\cP = \cQ$.
A classical result \cite[Prop.~3.4.4]{ethier1986markov} assures that if $(E,d)$ is a complete, separable and locally compact metric space, then $C_c(E)$ is separating (actually convergence determining). 
Notice that, being $C_c(E)$ separating for $E$, if $\cP, \cQ \in \PP(E)$ with $\cP \neq \cQ$ then there exists a function $h \in C_c(E)$ such that $\int_E h(x)\d \cP(x) \neq\int_E h(x) d\cQ(x)$. 
This means that the family 
\begin{equation}
\cF:= \left\lbrace \bar h: \bar h(\cP) = \int_E h(x)\d\cP(x),  h \in C_c(E) \right\rbrace \subset C(\PP(E); \R)
\end{equation}
separates points of $\PP(E)$ (endowed with the topology of the narrow convergence).

\medskip

Given $p \geq 1$, we define the $L^p$-Wasserstein distance
\begin{equation}
\label{eq:wasserstein}
W_p(\mu_0, \mu_1) := \inf \left\lbrace \int_{\R^d \times \R^d} |x-y|^p \,\d\gamma(x,y): \gamma \in \Pi(\mu_0,\mu_1) \right\rbrace;
\end{equation}
where $\Pi(\mu_0,\mu_1)$ is the set of the optimal plans:
\begin{equation*}
\Pi(\mu_0,\mu_1):= \lbrace \gamma \in \PP(\R^d \times \R^d): \gamma (B \times \R^d)= \mu_0(B), \gamma (\R^d \times B)= \mu_1(B) \quad \forall \, B \text{ Borel set in } \R^d \rbrace. 
\end{equation*}
The infimum in \eqref{eq:wasserstein} is always attained (and finite) if $\mu_0,\mu_1$ belong to the space
$\PP_p(\R^d)$ of Borel probability measure with finite $p$-moment.
Let us notice that $\PP_p(\R^d)$ endowed with the Wasserstein distance $W_p(\mu_0, \mu_1)$ is a complete and separable metric space.
A sequence $(\mu_n)_{n \in \N} \subset \PP_p(\R^d)$ converges to a limit $\mu \in \PP_p(\R^d)$, with respect to the Wasserstein distance $W_p$, i.e. $W_p(\mu_n,\mu) \to 0$, if 
\begin{equation}
\lim_{n \to \infty} \int_{\R^d} \varphi(x) \d \mu_n(x) = \int_{\R^d} \varphi(x) \d \mu(x), \qquad \forall \, \varphi \in C(\R^d), \;\sup_{x \in \R^d} \frac{\varphi(x)}{1 + |x|^p} <+ \infty
\end{equation} 
Notice that the class $C_c(\R^d)$ also separates points of $\PP_p(\R^d)$ endowed with the $p$-Wasserstein distance. 

If $X$ is a Polish space, Prokhorov theorem guarantees that a subset $\cK \subset \PP(X)$ is tight if and only if it is (relatively) compact. 
Moreover this is equivalent to the existnece of a function $\varphi: X \to \R_+$ with compact sublevels such that
\begin{equation}\label{UI}
\sup_{\nu \in \cK} \int_X \varphi(x) \d\nu(x) < +\infty.
\end{equation}

Consider now the subset of discrete measures $\PP^N(\R^d) \subset \PP_p(\R^d)$ given by 
\[ \PP^N(\R^d):= \left\lbrace \mu \in \PP_p(\R^d): \mu =  \frac{1}{N} \sum_{i=1}^N \delta_{x_i}, \; \text{ for some } x_i \in \R^d \right\rbrace  \]

Starting from a vector $\xx:= (x_1, \ldots, x_N) \in (\R^d)^N$ we can associate the measure $\mu^N[\xx]:= \frac{1}{N} \sum_{i=1}^N \delta_{x_i} \in \PP^N(\R^d)$  and
we refer to the map $\mu^N: (\R^d)^N \to \PP^N(\R^d)$ as to the empirical measure.
Notice that, given $\xx, \yy \in (\R^d)^N$ it holds 
\begin{equation}
W_p^p(\mu^N[\xx], \mu^N[\yy]) = \min_{\sigma} \frac{1}{N} \sum_{i=1}^N |x_i - \sigma(\yy)_{i}|^p,
\end{equation}
where $\sigma: (\R^d)^N \to (\R^d)^N$ denotes a permutation of coordinates of vectors in $(\R^d)^N$.

In the following we say that a map $\gG:\R^d\times (\R^d)^N\to \R^k$ is symmetric if 
\begin{displaymath}
  \gG(x,\yy)=\gG(x,\sigma(\yy))\quad\text{for every permutation  }\sigma:(\R^d)^N\to (\R^d)^N.
\end{displaymath}
Given a symmetric and continuous map $\gG^N$ we can associate
a function defined on measures $G:\R^d\times \PP^N(\R^d)\to \R^k$ by setting
\begin{equation}
  \label{eq:4}
  G(x,\mu[\yy]):=\gG(x,\yy).
\end{equation}

\medskip

If $X$ is a Polish space and $\cP, \cQ \in \PP(X)$ are two probability measures, the relative entropy of $\cQ$ with respect to $\cP$ is defined as:
\begin{equation}
\ent(\cQ | \cP) : = 
\begin{system}
\int_X \log (  \frac{d\cQ}{d\cP}) \d \cQ, \qquad \;\,\text{ if } \cQ \ll \cP; \\
+ \infty \qquad \qquad \quad \qquad  \text{ otherwise}.	
\end{system} 
\end{equation} 
Equivalently $\ent(\cQ | \cP) := \sup \lbrace \cQ(\phi) - \log(\cP(e^\phi)), \phi \in C_b(X) \rbrace$, from which the convexity of the map $\ent(\cdot, \cP)$ easily follows.
It is useful to recall the following basic inequality
\begin{equation}
\cQ(A) \leq \frac{\ent(\cQ | \cP) + \log 2}{\log(1 + \cP(A)^{-1})}.
\end{equation}
Moreover, if $Y$ is a Polish space and $f: X \to Y$ a measurable function  then
\begin{equation}\label{eq:ent_composiz.}
\ent(f_\sharp \cQ | f_\sharp \cP) \leq \ent(\cQ | \cP). 
\end{equation}

\subsection{Large deviations principle}

Large deviations estimates describe the limiting behaviour of a family of probability measure through the knowledge of a rate functional.
We refer to \cite{dembo1998large} for a general treatise of the subject.
Let us recall here the very definition of a Large Deviation Principle (in short LDP). 
\begin{definition}
Let $X$ be a Hausdorff topological space and $\cP_\varepsilon \in \PP(X)$ a family of probability measures on $X$. 
We say that $\cP_\varepsilon$ satisfies a good large deviation principle with speed $\beta_\varepsilon \downarrow 0$ and rate function $J: X \to [0,+\infty]$ if the following conditions are satisfied:
\begin{itemize}
\item[(i)] (Goodness) For every $a \geq 0$, the set $\lbrace x \in X: J(x) \leq a \rbrace$ is compact.
\item[(ii)] (Upper bound) For every closed set $C \subset X$
\begin{equation}\label{ub}
\limsup_{\varepsilon \to 0} \beta_{\varepsilon} \log \cP_\varepsilon(C) \leq - \inf_{x \in C} J(x). 
\end{equation}
\item[(iii)] (Lower bound) For every open set $U \subset X$
\begin{equation}
\liminf_{\varepsilon \to 0} \beta_{\varepsilon} \log \cP_\varepsilon(U) \geq - \inf_{x \in U} J(x). 
\end{equation}
\end{itemize} 	
\end{definition}

Establishing a LDP for the family $\cP_\varepsilon$ gives a precise formulation of (logarithmic) asymptotic bounds of the form $\cP_\varepsilon(B) \asymp \exp(-\beta^{-1} \inf_B J)$.
The validity of such a result implies that $\inf_{x \in X} J(x) = 0$, where the zero level set is not empty thanks to the goodness of the rate functional.
Notice that, in the special case $\lbrace \bar x \rbrace = \lbrace x : J(x)  =0 \rbrace$, a LDP implies the law of large numbers $\cP_\varepsilon \to \delta_{\bar x}$.

A fairly classical strategy to get \eqref{ub} consists in showing it first for compact sets and subsequently prove that the most of the probability mass is concentrated on compact sets.
To this aim, the notion of exponential tightness comes into play:
\begin{definition}
Let $X$ be a Hausdorff topological space and a sequence $\beta_\varepsilon \downarrow 0$.
We say that a sequence of probability measures $\cP_\varepsilon \in \PP(X)$ is exponentially tight with speed $\beta_\varepsilon$ if there exists a sequence of compacts $K_l$ in $X$ such that 
\begin{equation}
\lim_{l \to + \infty} \limsup_{\varepsilon \to 0} \beta_\varepsilon \log \cP_{\varepsilon}(X \setminus K_l) = -\infty.
\end{equation} 
\end{definition} 

Notice that exponential tightness of the family $\cP_\varepsilon$ is not a priori necessary for the formulation of a LDP with good rate functional. 
Nonetheless, if the family $\cP_\varepsilon$ is exponentially tight and satisfies a LDP lower bound for every open sets, then the rate functional is automatically good.

Let us now recall a characterization of exponential tightness in the space of continuous functions $C([0,T];\R)$, inspired by \cite{Billingsley1999convergence}, which will be useful in the sequel.
Given a function $u \in C([0,T]; \R)$ we denote by $\omega(u;\delta)$ its continuity modulus:
\begin{equation}\label{cont_modulus}
\omega(u;\delta):= \sup_{\substack{t,s \in [0,T] \\ |t-s| \leq \delta}} |u_t - u_s|.
\end{equation}

\begin{proposition}\label{p:exp_tight_R}
A sequence $\cP_\varepsilon \in \PP(C([0,T];\R))$ of probability measures is exponentially tight with speed $\beta_\varepsilon$ iff the following  two conditions hold:
\begin{itemize}
\item[(a)]  $\lim_{l \to \infty} \limsup_{\varepsilon \to 0} \beta_\varepsilon \log \cP_\varepsilon(u: \, |u_0| > l) = -\infty$;
\item[(b)] $\lim_{\delta \to 0} \limsup_{\varepsilon \to 0} \beta_\varepsilon \log \cP_\varepsilon(u: \,\omega(u,\delta) > \zeta) = -\infty, \;$ for any $\zeta > 0$.
\end{itemize}
\end{proposition}

\subsection{Stochastic currents and continuity equation}

The notion of (Stratonovich) stochastic currents were discussed in \cite{flandoli2005stochastic} with the attempt to investigate the links between deterministic currents and the theory of rough paths.
One of the main interest of the paper is the pathwise regularity of stochastic integrals of the form:
\begin{equation}
J(\eta) = \int_0^T \eta(t,X_t) \circ \d X_t,
\end{equation} 
where $\eta: Q \to \R^d$ is a compactly supported smooth vector field and $X_t$ is a semimartingale.
When $\eta$ does not depend on time, the authors  in \cite{flandoli2005stochastic} showed that the map $\eta \mapsto J(\eta)$ defines with probability one a linear functional on $H^{s+1}(\R^d; \R^d)$ for $s > d/2+1$. 
The extension of this result to the time dependent case is the content of the following theorem. 
Let us use the notation $\hH^s:= H^{s_1}([0,T]; H^{s_2}(\R^d;\R^d))$, where $s = (s_1,s_2)$.
\begin{theorem}\label{t:reg_stoch_current}
Let $X_t = V_t + M_t$ be a semimartingale with values in $\R^d$ and $\eta: Q\to \R^d$ be a smooth function with compact support.
Then, given $s_1 \in (1/2,1)$ and $s_2 \in (\frac{d+2}{2}, +\infty)$,  the map $\eta \mapsto J(\eta)$ has a pathwise realization $\cJ$ such that $[\cJ(\omega)](\eta) =  [J(\eta)](\omega)$ and 
\begin{equation}
\cJ(\omega) \in \hH^{-s} \qquad \P\text{-a.s.}
\end{equation} 
\end{theorem}
\begin{proof}
see Appendix.
\end{proof}
Let now $(\mu_t)_{t \in [0,T]}$ be a family of probability measures satisfying the continuity equation
\begin{equation}\label{ec}
\partial_t \mu_t + \nabla \cdot J  = 0,
\end{equation}
where the term $J$ represents the current.
Here we collect some properties of solutions to the above equation, emphasizing the link with the regularity of $J$.
Let us start with a general definition.
\begin{definition}\label{def:continuity:distrib}
Let $(\mu_t)_{t \in (0,T)}$ be a family of probability measures on $\R^d$ and $J \in \hH^{-s}$ be a space-time distribution with $s=(s_1,s_2) $, $s_1 >1/2, s_2 >d/2$.
We say the $(\mu, J)$ is a distributional solution to \eqref{ec} if 
\begin{equation}
\int_0^T \int_{\R^d} \partial_t \phi(t,x) \d\mu_t(x) \d t + J(\nabla \phi)  = 0,
\end{equation}
for every $\phi \in C^\infty_c(Q)$. 
\end{definition}

We will show in Lemma \ref{l:constraint_N} that the empirical measure and the stochastic current associated to the particles system \eqref{eq:state_intro} exactly fits into this framework, thanks to the pathwise regularity shown in Theorem  \ref{t:reg_stoch_current}.

\begin{remark}\label{r:boundary_data}
To include initial ot final constraint, say $\mu_0, \mu_T$, in the definition of solution we can use test functions $\phi \in C^\infty_c([0,T] \times \R^d)$ so that
\begin{equation}
\int_{\R^d} \phi(T,x)\d\mu_T(x) - \int_{\R^d} \phi(0,x)\d\mu_0(x) =  \int_0^T \int_{\R^d} \partial_t \phi(t,x) \d\mu_t(x) \d t + J(\nabla \phi),
\end{equation}
for every $\phi \in C^\infty_c([0,T] \times \R^d)$.
\end{remark}

Let us now concentrate on a more specific situation.
Specifically, suppose there exists a Borel vector field $\vv: (t,x) \mapsto \vv_t(x) \in \R^d$ such that 
\begin{equation}\label{integr-v}
\int_0^T \int_{\R^d} |\vv_t(x)|^2 \d\mu_t(x) \d t < +\infty,
\end{equation} 
and $J$ can be written in the form 
\begin{equation}\label{vector_field}
J(\eta) = \int_0^T \int_{\R^d} \eta(t, x) \cdot \vv(t, x) \d \mu_t(x) \d t, \qquad \forall \, \eta \in C^\infty_c(Q; \R^d).
\end{equation} 
In this case the distributional formulation in Definition \ref{def:continuity:distrib} reads 
\begin{equation}\label{ec_v}
\int_0^T \int_{\R^d}\big( \partial_t \phi(t,x)  + \vv_t(x) \cdot \nabla_x \phi(t,x) \big) \d\mu_t(x) \d t = 0, \qquad \forall \,\phi \in C^\infty_c(Q).
\end{equation} 
and there is a by now classical connection between solutions to \eqref{ec_v} and absolutely continuous curves $\mu: [0,T] \to \PP_2(\R^d)$, see e.g. \cite[Thm.~8.3.1]{ambrosio2008gradient}.

More precisely, given a curve $t \mapsto \mu_t \in \operatorname{AC}^2([0,T]; \PP_2(\R^d))$ it is convenient to  define a space-time measure
$\tilde \mu:= \frac{1}{T}\int_0^T \delta_t \otimes \mu_t \, \d t \in \PP_2(Q)$ satisfying
\[ \int_Q h(t,x) \d \tilde \mu (t,x) = \frac{1}{T}\int_0^T \int_{\Rd} h(t,x) \d\mu_t(x) \d t, \qquad \forall h \in C_b(Q).\]
Then we can find a (minimal) Borel vector field $\vv \in L^2(Q, \tilde \mu; \Rd)$ (i.e. $\int_Q |\vv(t,x)|^2 \d \tilde \mu(t,x) < +\infty$) such that $J = \vv \tilde \mu \ll \tilde \mu$ is a vector measure and solves in a distributional sense 
\begin{equation}\label{ec_space-time}
\partial_t \tilde \mu + \nabla \cdot J = 0,
\end{equation}
which is equivalent to \eqref{ec_v}. 
On the other hand, when $\mu$ is a solution to \eqref{ec_v} with a velocity field satisfying \eqref{integr-v} then there exists a representative $t \mapsto \mu_t \in \PP_2(\Rd)$, still denoted with $\mu_t$, belongings to $ \operatorname{AC}^2([0,T]; \PP_2(\R^d))$.

\section{Statement of the problem and main result}

Consider $N$ particles $(x_1, \ldots, x_N)$ whose dynamics is described by the following system of SDEs:
\begin{equation}
  \label{eq:state}
\d x_i^{N,\varepsilon}(t)=\fF(x_i^{N,\varepsilon}(t),\xx^{N,\varepsilon}(t))\d t + \sqrt{\varepsilon}\d W_i(t) ,\quad i=1,\cdots, N,\qquad
  \xx^N(0)=\xx^N_0.
\end{equation}
where the map $\fF:\R^d\times (\R^d)^N\to \R^d$ models the mean field interaction and $(W_i(t))_{t \in [0,T]}$ are $d$-dimensional independent Brownian motions.
Given a time horizon $T > 0$, with a little abuse of notation we refer to the empirical measure as to 
\[\mu^N: C([0,T];(\R^d)^N) \to C([0,T];\PP(\R^d)), \qquad \mu^N[x](t) = \mu^N[x_t],  \]  
and we will denote by $\mu^{N,\varepsilon}_t$ the image measure associated to a solution of \eqref{eq:state} for every $t \in [0,T]$:
\begin{equation}\label{empirical:measure}
\mu^{N,\varepsilon}_t := \frac{1}{N}\sum_{i=1}^N \delta_{x^{N,\varepsilon}_i(t)}.
\end{equation}
By construction, $\mu^{N,\varepsilon}_t$ is a random measure for every $t \in [0,T]$. 
We denote by $\P^{N,\varepsilon}$ the law of the $N$-dimensional system $\P^{N,\varepsilon}:= \law(\xx^{N,\varepsilon}) = (\xx^{N,\varepsilon})_\sharp \P$ and by $\cP^{N,\varepsilon} \in \PP(C([0,T]; \PP(\R^d)))$ the measure $\cP^{N,\varepsilon}:= (\mu^N)_\sharp \P^{N,\varepsilon}$ induced by the empirical measure. 
The probability spaces we are dealing with are the following
\begin{equation}
(\Omega, \P)  \xrightarrow{(x^{N,\varepsilon})_\sharp} \left(  C([0,T]; (\R^d)^N), \P^{N,\varepsilon}\right) \xrightarrow{(\mu^N)_\sharp} \left(  C([0,T]; \PP(\R^d)), \cP^{N,\varepsilon}\right).  
\end{equation}

To complement the information contained in the random measures $( \mu^{N,\varepsilon}_t)_{t \in [0,T]}$ we introduce for every $\eta \in C^\infty_c(Q; \R^d)$ the stochastic current
\begin{equation}\label{stoch:current}
J^{N,\varepsilon}(\eta) := \frac{1}{N}\sum_{i=1}^N \int_0^T \eta(t,x_i^{N,\varepsilon}(t)) \circ \d x^{N,\varepsilon}_i(t).
\end{equation}
From the general theory on stochastic currents (see also Theorem \ref{t:reg_stoch_current}) the stochastic integral defined in 
\eqref{stoch:current} has a pathwise realization $\cJ^{N,\varepsilon}$, for every $N \in \N$ and $\varepsilon >0$. 
Moreover $\cJ^{N\varepsilon} \in \hH^{-s}$, for every $s = (s_1,s_2) \in (\frac{1}{2},1) \times (\frac{d+2}{2}, +\infty)$.

The objective of the paper is to investigate the behaviour of the system \eqref{eq:state} as the number of particle tends to infinity and, simultaneously,  in the small-noise regime. 
More precisely, denoting by $\cX$ the following space 
\[\cX: = C([0,T];\PP_1(\R^d)) \times H^{-s_1}\left((0,T); H^{-s_2}(\R^d;\R^d)\right),\]
we are interested in large deviations properties in the joint limit $\varepsilon \downarrow 0, N \uparrow +\infty$ for the probability measures $\P^{N,\varepsilon} \circ (\mu^N, J^{N,\varepsilon})^{-1} \in \PP(\cX)$.
We endow $C([0,T];\P_1P(\R^d))$ with the uniform 1-Wasserstein topology and $H^{-s_1}\left((0,T); H^{-s_2}(\R^d;\R^d)\right)$ with the weak-topology.
Precisely, a sequence $(\mu^n, J^n)$ converges to $(\mu,J)$ in $\cX$ when
\[\lim_{n \to +\infty} \sup_{t \in [0,T]} W_1(\mu^n_t,\mu_t) = 0, \qquad \text{ and } \qquad  J^n(\eta) \to J(\eta), \; \forall \, \eta \in C_c^\infty(Q;\Rd). \] 
Notice that $\cX$ equipped with the above topology is not metrizable, hence not a Polish space. 
Nonetheless, to state a large deviation principle it is enough to have a Hausdorff topological space. 
In our setting $\cX$ is actually a Tychonoff space with metrizable compacts ($\hH^s$ is indeed separable and reflexive).

\smallskip

In the following, to emphasize the differences in obtaining the lower and the upper bounds in the LDP, we prefer to keep the hypotheses separated.
The entire LDP holds a fortiori under the stronger assumptions of the lower bound.  
The first set of assumptions on the interaction field $F$ and on the set of initial conditions are the following
\begin{hypothesis}[Interaction (upper bound)]\label{h1}
The function $\fF:\R^d\times \PP_1(\R^d)\to \R^d$ is continuous and there exists a constant $A\geq 0$ such that
\begin{equation} 
|\fF(x,\mu)|\le A\left(1 + |x|+\int_{\Rd} |x|\,\d\mu(x)\right).
\end{equation}
When $\fF:\R^d\times \PP^N(\R^d)\to \R^d$ we can alternatively consider it as a symmetric function $\fF: \R^d\times (\R^d)^N\to \R^d$, thanks to the identification \eqref{eq:4}. 
\end{hypothesis}

\begin{hypothesis}[Initial data (upper bound)]\label{h2}
The initial distribution $\mu_0^N: = \frac{1}{N}\sum_{i=1}^N \delta_{x_i^N(0)}$ is deterministic for every $N \in \N$ and  
$\sup_{N \in \N} \int_{\R^d} |x|^2 \d\mu^N_0 < +\infty$. Moreover there exists $\mu_0 \in \PP_1(\R^d)$  and $W_1(\mu_0^N, \mu_0) \to 0$ as $N \uparrow +\infty$.
\end{hypothesis}
\noindent Hypothesis \ref{h2} coincides with the law of large numbers for the deterministic initial conditions, necessary for the convergence of the empirical measure associated to the system. 

The second set of assumptions is as follows 

\begin{hypothesis}[Interaction (lower bound)]\label{h.lip}
Suppose there exists $K > 0$ such that for every $t \in [0,T]$, $x,x' \in \R^d$, $\mu, \mu' \in \PP_1(\R^d)$ we have
\begin{equation}
\left| \fF(x,\mu) - \fF(x', \mu') \right| \leq L \left(  |x-x'| + W_1(\mu,\mu') \right).
\end{equation}
When $\fF:\R^d\times \PP^N(\R^d)\to \R^d$ we can alternatively consider it as a symmetric function $\fF: \R^d\times (\R^d)^N\to \R^d$, thanks to the identification \eqref{eq:4}. 
\end{hypothesis}

\begin{hypothesis}[Initial data (lower bound)]\label{h.comp}
The initial distribution $\mu_0^N: = \frac{1}{N}\sum_{i=1}^N \delta_{x_i(0)}$ is deterministic for every $N \in \N$ with uniformly compact support. Moreover $\mu_0 \in \PP_1(\R^d)$ and $W_1(\mu_0^N, \mu_0) \to 0$ as $N \uparrow +\infty$.
\end{hypothesis}
The Lipschitz character of the interaction term in Hypothesis \ref{h.lip} provides uniqueness of solutions to \eqref{eq:state} (see also Theorem \ref{ex_uniq_vlasov} in the Appendix for what concerns the continuity equation) and it is crucial in the proofs of Propositions \ref{p:wasserstein_esitmate} and \ref{p:diff_current_esitmate}.
On the other hand, the compact support of the initial conditions is not directly used but it is needed to profitably apply the convergence result in \cite{fornasier2018mean}. 

\begin{remark} (An example of interaction):
Given a continuous function $H:\R^d\to \R^d$ satisfying
  \[
   |H(x) - H(y)| \leq L_H |x-y|, \quad \forall \, x,y \in\R^d
  \]
   for some $L_H > 0$, we can set
  \begin{equation}
    \fF(x,\yy):=\frac 1N\sum_{j=1}^N H(x-y_j)=
    \int_{\R^d} H(x-y)\,\d\mu[\yy](y)
  \end{equation}
  and
   \begin{equation}
    \label{eq:9bis}
    \fF(x,\mu):= \int_{\R^d} H(x-y)\,\d\mu(y).
  \end{equation}
More specifically, if  $H=-\nabla W$ for an even function $W\in C^1(\R^d)$, the system \eqref{eq:state} can be viewed as a stochastic perturbation of  the gradient flow of the interaction energy $\cW:(\R^d)^N\to\R$ defined by
  \begin{equation}
    \cW(\xx):=\frac 1{N^2}\sum_{i \leq j}^NW(x_i-x_j)
  \end{equation}
  with respect to the norm $\|\xx\|^2=\frac 1N\sum_{i=1}^N|x_i|^2$.
  \end{remark}

A first result on the convergence of the finite dimensional stochastic system \eqref{eq:state} toward  a purely deterministic  evolution of measures is contained in the following 

\begin{theorem}\label{t:Limit}
If Hypotheses \ref{h.lip} and \ref{h.comp} hold, there exists a unique strong solution $\xx^{N,\varepsilon}$ to the system \eqref{eq:state}.
The associated empirical measure and stochastic current $(\mu^{N,\varepsilon}, J^{N,\varepsilon})$ admits a limit $(\mu,J)$ as $N \uparrow +\infty$ and $\varepsilon \downarrow 0$ in the following sense
\begin{equation}
\begin{split}
\lim_{\substack{\varepsilon \to 0\\ N \to +\infty}} \P^{N, \varepsilon} \left( \sup_{t \in [0,T]} W_1(\mu^{N,\varepsilon}_t, \mu_t) > \delta \right) &=0 \\
\lim_{\substack{\varepsilon \to 0\\ N \to +\infty}} \P^{N, \varepsilon} \left( \left| J^{N,\varepsilon}(\eta) - J(\eta) \right|^2 > \delta \right) &=0, \qquad \forall \; \eta \in C_c^{\infty}(Q; \R^d),
\end{split}
\end{equation}
where the pair $(\mu,J)$ is the unique distributional solution to 
\begin{equation}
\partial_t \mu_t + \nabla \cdot J_t = 0, \qquad J_t = \fF(\cdot,\mu_t) \mu_t, \qquad \mu(0,\cdot) = \mu_0(\cdot).
\end{equation} 
\end{theorem}

Here we are interested in a more precise asymptotic analysis of the behaviour of the pair $(\mu^{N,\varepsilon}, J^{N,\varepsilon})$.
The main result of the paper describes the rate at which the probability of rare events occurs and it is formulated as a LDP for the probability measures $\P^{N,\varepsilon} \circ (\mu^N, J^{N,\varepsilon})^{-1} \in \PP(\cX)$ as $N \uparrow +\infty$ and $\varepsilon \downarrow 0$. Recall that we use the notation $\tilde \mu$ for the space-time measure $\tilde \mu:= \int_0^T \delta_t \otimes \mu_t \, \d t$.

\begin{theorem}[LDP]\label{t:LDP}
Let $I: \cX \to [0,+\infty]$ be the functional (given in a variational formulation)
\begin{equation}
\begin{split}
I(\mu,J)  = \sup_{\eta \in C^{\infty}_c(Q;\R^d)} \Big\lbrace  &J(\eta)  - \int_0^T \langle \mu_t, \eta(t, \cdot) \cdot \fF(\cdot, \mu_t) \rangle \d t \\
& - \frac{1}{2}\int_0^T \langle \mu_t, \left| \eta(t,\cdot)  \right|^2 \rangle \d t \quad \Big| \;\partial_t \mu_t  + \nabla \cdot J = 0; \;\;  \mu(0)=\mu_0 \Big\rbrace
\end{split} 
 \end{equation}
 and define $\cX_l:= \lbrace (\mu,J) \in \cX: \int_{\R^d}|x|^2 \d\mu_0(x) \leq l \rbrace$.
Then
\begin{itemize}
\item[(i)] If Hypothesis \ref{h1} holds the rate functional $I: \cX_l \to [0,+\infty]$ is good. Moreover, when $I (\mu,J)<+\infty$ there exists $\hh \in L^2(Q,\tilde \mu;\R^d)$ such that $J = \hh \tilde \mu$, 
\begin{equation}
I(\mu,J) = \frac{1}{2} \int_0^T \langle \mu_t, |\hh(t,\cdot)|^2 \rangle \d t
\end{equation}
and the pair $(\mu,\hh)$ satisfies  $\partial_t \mu_t + \dive \left( \fF(\cdot, \mu_t) + \hh \right)\mu_t= 0$ with $\mu(0)=\mu_0$, in a distributional sense.
\item[(ii)]  
Under Hypotheses \ref{h1} and \ref{h2}, the sequence of probability measures $\P^{N,\varepsilon} \circ (\mu^N, J^{N,\varepsilon})^{-1} \in \PP(\cX)$ satisfies a LD upper bound on $\cX$ with speed $\varepsilon/ N$
 and rate functional $I: \cX \to [0,+\infty]$: 
 
 for every closed set $\sfC \subset \cX$
 \begin{equation}
\limsup_{\substack{\varepsilon \to 0\\ N \to +\infty}} \frac{\varepsilon}{N} \log \P^{N,\varepsilon} ((\mu^{N}, J^{N,\varepsilon}) \in \sfC) \leq - \inf_{(\mu, J) \in \sfC} I(\mu, J).
\end{equation}
 \item[(iii)] If Hypotheses \ref{h.lip} and \ref{h.comp} hold, the sequence of probability measures $\P^{N,\varepsilon} \circ (\mu^N, J^{N,\varepsilon})^{-1} \in \PP(\cX)$ satisfies a LD lower bound on $\cX$ with speed $\varepsilon/ N$
 and rate functional $I: \cX \to [0,+\infty]$: 
 
 for every open set $\sfO \subset \cX$
 \begin{equation}
\liminf_{\substack{\varepsilon \to 0\\ N \to +\infty}} \frac{\varepsilon}{N} \log \P^{N,\varepsilon} ((\mu^{N}, J^{N,\varepsilon}) \in \sfO) \geq - \inf_{(\mu, J) \in \sfO} I(\mu, J).
\end{equation}
\end{itemize}
\end{theorem}

The proofs of the above theorems are postponed at the end of Section 5.

\begin{remark}
By the contraction principle, see e.g. \cite[Thm.~ 4.2.1]{dembo1998large}, we can derive a large deviation principle for the measures $\cP^{N,\varepsilon} \in \PP(C([0,T]; \PP(\R^d)))$ with rate functional $\bar I: C([0,T]; \PP(\R^d)) \to [0,+\infty]$ given by
\begin{equation}
\bar I(\mu) = \inf \lbrace I(\mu,J), J \in \hH^{-s} \rbrace.
\end{equation}
Notice that in the regime $\varepsilon >0$, given a measure $\mu^\varepsilon$ there exists a unique current $J^\varepsilon$ for which the McKean-Vlasov PDE is satisfied (thanks to the parabolic character of the equation). In the limit $\varepsilon \downarrow 0$ this is no longer true and the application of the contraction principle is not trivial.
Once more, this naturally requires the formulation of the above LDP for the pair measure-current.  
\end{remark}

\section{Large deviations upper bound}

This section is devoted to the analysis of the large deviation upper bound for the family of probability measures $\P^{N,\varepsilon} \circ (\mu^N, J^{N,\varepsilon})^{-1} \in \PP(\cX)$. 
Hereafter, given $(\mu,J) \in \cX$ we denote by $I: \cX \to [0,+\infty]$ the functional  
\begin{equation}\label{rate_variational}
\begin{split}
I(\mu,J)  = \sup_{\eta \in C^\infty_c(Q; \R^d)} \Big\lbrace  &J(\eta)  - \int_0^T \langle \mu_t, \eta(t, \cdot) \cdot \fF(\cdot, \mu_t) \rangle \d t \\
& - \frac{1}{2}\int_0^T \langle \mu_t, \left| \eta(t,\cdot)  \right|^2 \rangle \d t \quad \Big| \;\partial_t \mu_t  + \dive J = 0; \;\;  \mu(0)=\mu_0 \Big\rbrace,
\end{split}
\end{equation}
where the constraint has to be intended in a distributional sense (see Definition \ref{def:continuity:distrib}).
%

Let us firstly investigate the relation between the empirical measure \eqref{empirical:measure} and the stochastic current  \eqref{stoch:current}.
Exploiting the independence of the brownian motions in the dynamic \eqref{eq:state} and applying Theorem \ref{t:reg_stoch_current} to the stochastic current $J^{N,\varepsilon}$ defined in \eqref{stoch:current} we get a pathwise realization $\cJ^{N,\varepsilon}$, for every $N \in \N$ and $\varepsilon >0$ with 
$\cJ^{N\varepsilon} \in \hH^{-s}$, $s = (s_1,s_2) \in (\frac{1}{2},1) \times (\frac{d+2}{2}, +\infty)$.

Furthermore, the the pair $(\mu^{N,\varepsilon}, J^{N,\varepsilon})$ satisfies a continuity equation as it is shown in the next Lemma.
\begin{lemma}\label{l:constraint_N}
Given $\mu^N_0 \in \PP(\R^d)$, the pair $(\mu^{N,\varepsilon}, J^{N, \varepsilon})$ defined in \eqref{empirical:measure}-\eqref{stoch:current} 
is a distributional solution in the sense of Definition \ref{def:continuity:distrib} (see also Remark \ref{r:boundary_data}) to the following 
\begin{equation}\label{eq:cont_discr}
\begin{system}
\partial_t \mu_t^{N,\varepsilon} + \nabla \cdot J^{N,\varepsilon}= 0, \qquad \P\text{-a.e.}, \\
\mu^{N,\varepsilon}(0) = \mu^N_0 
\end{system}
\end{equation}
for every $N \in \N$ and $\varepsilon >0$. 
\end{lemma}
\begin{proof}
The proof follows by the application of It\^o's formula to the system \eqref{eq:state} with $\phi \in C_c^\infty([0,T) \times \R^d)$:
\begin{equation}
\phi(0, x_i^{N}(0)) +  \int_0^T \partial_t \phi(s, x^{N,\varepsilon}_i(s)) \d s  + \int_0^T \nabla \phi(s, x^{N,\varepsilon}_i(s)) \circ \d x^{N,\varepsilon}_i(s) = 0, 
\end{equation}
averaging with respect to $N$ we get 
\begin{equation}
\int_{\R^d} \phi(0)\d\mu_0^{N} + \int_0^T \int_{\R^d} \partial_s \phi(s,x) \d \mu^{N,\varepsilon}_s \d s  + J^{N,\varepsilon}\left( \nabla \phi \right) = 0.  
\end{equation}
From Theorem \ref{t:reg_stoch_current} we know that $J^{N,\varepsilon}$ has a pathwise realization $\cJ^{N,\varepsilon} \in \hH^{-s}$, where $s = (s_1,s_2) \in (\frac{1}{2},1) \times (\frac{d+2}{2}, +\infty)$.
Hence, equation \eqref{eq:cont_discr} holds for $\P$-a.e. $\omega \in \Omega$. 
\end{proof}

Let us now exhibit a direct proof of the LD upper bound for compact sets, via a specific exponential tilt of the measures $\left(\P^{N,\varepsilon}\right)_{N,\varepsilon}$. 
The general case (for closed sets) will be obtained exploiting the exponential tightness of the measures.

\begin{proposition}\label{p:upper_bound_compacts}
Under Hypotheses \ref{h1}, \ref{h2}, for every compact $\sfK \subset \cX$ it holds 
\begin{equation}\label{eq:limsup}
\limsup_{\substack{\varepsilon \to 0\\ N \to +\infty}} \frac{\varepsilon}{N} \log \P^{N,\varepsilon} ((\mu^{N}, J^{N,\varepsilon}) \in \sfK) \leq - \inf_{(\mu, J) \in \sfK} I(\mu, J).
\end{equation}
\end{proposition}

\begin{proof}
Given $\sfB \subseteq \cX$ a Borel set, we want to estimate the quantity $\P^{N,\varepsilon}((\mu^N,J^{N,\varepsilon}) \in \sfB)$.  
Fix $\eta \in C_c^\infty(Q; \R^d)$ and define the martingale	
\begin{equation}
M_t^\eta := \frac{\sqrt{\varepsilon}}{N} \sum_{1=1}^N \int_0^t \eta(s,x^{N,\varepsilon}_i(s)) \d W_i(s),
\end{equation}
with quadratic variation
\begin{equation}
[M^\eta]_t  = \frac{\varepsilon}{N} \int_0^t \langle \mu^{N,\varepsilon}_s, |\eta(s,\cdot)|^2 \rangle \d s.
\end{equation}
Then, the corresponding stochastic exponential
$\exp(M_t^\eta - \frac{1}{2} [M^\eta]_t)$
is a martingale with 
\begin{equation}\label{eq:supermart}
\E^{N,\varepsilon}\left[ \exp\left(M_t^\eta - \frac{1}{2} [M^\eta]_t\right)\right] \leq 1, \qquad \forall \, t \in [0,T]
\end{equation} 
Moreover, using the definition of the stochastic current given in \eqref{stoch:current}, we can write
\begin{equation}\label{eq:rel_M-J}
M_T^\eta = J^{N,\varepsilon}(\eta) - \int_0^T \langle \mu^{N,\varepsilon}_s, \eta(s,\cdot) \cdot \fF(\cdot, \mu^{N,\varepsilon}_s) \rangle \ ds + \varepsilon R^{N,\varepsilon},
\end{equation}
where we denoted by $R^{N,\varepsilon}:= \int^T_0 \langle \mu^{N,\varepsilon}_s, \dive(\eta)(s, \cdot) \rangle \d s$ the It\^o correction term.

Now, for any $\eta \in C_c^\infty(Q; \R^d)$ and $\phi \in C_c^\infty([0,T] \times \R^d)$ we define the maps $I^{\eta}_1, I^\phi_2: \cX \to [0,+\infty]$ as
\begin{equation*}
\begin{split}
I^{\eta}_1(\mu,J) &:= J(\eta) - \int_0^T \langle \mu_s, \eta(s,\cdot)\cdot \fF(\cdot,\mu_s) \rangle \d s - \frac{1}{2}\int_0^T \langle \mu_s, |\eta(s,\cdot)|^2 \rangle \d s, \\
I^\phi_2(\mu,J) &:= \langle \mu(T),\phi(T) \rangle - \langle \mu_0, \phi(0) \rangle - \int_0^T \langle \mu_s, \partial_s \phi(s) \rangle \d s - J(\nabla \phi),
\end{split}
\end{equation*}
where for any $N \in \N$, $\varepsilon > 0$, it is worth notinicing that $ I^\phi_2(\mu^{N,\varepsilon}, J^{N,\varepsilon})=0$ (see Lemma \ref{l:constraint_N}). 
Replacing $\eta$  with $\frac{N}{\varepsilon}\eta$ we get
\begin{equation}
\begin{split}
\P^{N,\varepsilon} \left(\left(\mu^{N,\varepsilon}, J^{N,\varepsilon}\right) \in \sfB \right) &= \E^{N,\varepsilon}\left( 1_{\sfB}(\mu^{N,\varepsilon}, J^{N,\varepsilon}) \right) \\
&=\E^{N,\varepsilon} \left( e^{-M_T^\eta + \frac{1}{2}[M^\eta]_T} e^{M_T^\eta - \frac{1}{2}[M^\eta]_T} e^{I^\phi_2(\mu^{N,\varepsilon}, J^{N,\varepsilon})} 1_{\sfB}(\mu^{N,\varepsilon}, J^{N,\varepsilon}) \right)  \\
&\leq \sup_{(\mu,J) \in \sfB} \left[e^{-\frac{N}{\varepsilon} \left( I_1^\eta(\mu,J) + I_2^\phi(\mu,J) +\varepsilon R^{N,\varepsilon} \right)}\right]\E^{N,\varepsilon} \left( e^{M_T^\eta - \frac{1}{2}[M^\eta]_T}  1_{\sfB}(\mu^{N,\varepsilon}, J^{N,\varepsilon}) \right) \\
&\leq \sup_{(\mu,J) \in \sfB} e^{-\frac{N}{\varepsilon} \left( I_1^\eta(\mu,J) + I_2^\phi(\mu,J) +\varepsilon R^{N,\varepsilon} \right)},
\end{split}
\end{equation}
thanks to relations \eqref{eq:supermart} and \eqref{eq:rel_M-J}.

Taking the $\limsup$ as $\varepsilon \downarrow 0$ and $N \uparrow + \infty$, the term $\varepsilon |R^{N,\varepsilon}|$ vanishes and we can optimize in $\eta, \phi$ to get  
 \begin{equation}
 \begin{split}
 \limsup_{\substack{\varepsilon \to 0\\ N \to +\infty}} \frac{\varepsilon}{N} \log \P^{N, \varepsilon} \left( (\mu^{N, \varepsilon},J^{N, \varepsilon})  \in B \right) &\leq -\inf_{(\mu,J) \in B} \left( I_1^\eta(\mu,J) + I^\phi_2(\mu,J) \right) \\
 &\leq - \sup_{\substack{\eta \in C^\infty_c(Q; \R^d)\\ \phi \in C^\infty_c([0,T)\times \R^d)}} \inf_{(\mu,J) \in B} \left( I_1^\eta(\mu,J) + I^\phi_2(\mu,J) \right).
\end{split}
\end{equation}
The lower semicontinuity of the map $(\mu,J) \mapsto [ I^\eta_1(\mu,J) + I^\phi_2(\mu,J)]$ (seen as a map from from $\cX$ to $\R$) allow for the application  of the minmax Lemma, see \cite[App. 2, Lemmata 3.2 and 3.3]{kipnis1999scaling}, whence, for every compact set $\sfK \in \cX$,
\begin{equation}
\begin{split}
\limsup_{\substack{\varepsilon \to 0\\ N \to +\infty}} \frac{\varepsilon}{N} \log \P^{N, \varepsilon} \left( (\mu^{N, \varepsilon},J^{N, \varepsilon})  \in \sfK \right) &\leq - \inf_{(\mu,J) \in \sfK}  \sup_{\substack{\eta \in C^\infty_c(Q; \R^d)\\ \phi \in C^\infty_c([0,T)\times \R^d)}} \left( I_1^\eta(\mu,J) + I^\phi_2(\mu,J) \right)\\
&= - \inf_{(\mu,J) \in \sfK}  \sup_{\eta \in C^\infty_c(Q; \R^d)} \left( I_1^\eta(\mu,J) \; \big| \; \partial_t \mu_t  + \dive J = 0; \;\;  \mu(0)=\mu_0 \right),
\end{split}
\end{equation} 
since the sup in $\phi$ takes value $0$ if the constraint is satisfied and $+ \infty$ otherwise.
This implies that \eqref{eq:limsup} holds with rate function
\begin{equation}
\begin{split}
I(\mu,J) = \sup_{\eta \in C^\infty_c(Q)} \left\lbrace  I^{\eta}(\mu,J) \; \Big| \; \partial_t \mu_t  + \dive J = 0; \;\;  \mu(0)=\mu_0 \right\rbrace,
\end{split}  	 
\end{equation}
which is the required result.
\end{proof}

\subsection{Exponential tightness}
This section is devoted to the exponential tightness of the family $\P^{N,\varepsilon}\circ (\mu^N, J^{N, \varepsilon})^{-1} \in \PP(\cX)$ for which we investigate separately the two components $\P^{N,\varepsilon} \circ \left( \mu^N\right)^{-1} \in \PP\left( C([0,T] ;\PP(\R^d)) \right)$  and $\P^{N,\varepsilon} \circ \left( J^{N,\varepsilon}\right)^{-1} \in \PP\left( \hH^{-s} \right)$.   
For what concerns the family $\P^{N, \varepsilon} \circ \left(\mu^{N}\right)^{-1}$ a tightness criterium was first established by Jakubowski in \cite{jakubowski1986skorokhod}. 
Here we need a finer results, taking into account the exponential decay out of compact sets, which can be stated as follows.
\begin{theorem}\label{t:criterium_exp_tight}
Let $E$ be a completely regular topological space with metrizable compacts. 
If we endow the space $C([0,T];E)$ with the compact-open topology then a family of probability measures $\cP_\varepsilon \in \PP(C[0,T];E)$ is exponentially tight with speed $\beta_\varepsilon$ if the following conditions are satisfied:
\begin{itemize}
\item[(i)] there exists a sequence of compacts $K_l$ in $E$ such that 
\begin{equation}
\lim_{l \to + \infty} \limsup_{\varepsilon \to 0} \beta_\varepsilon \log \cP_{\varepsilon}(\exists \, t: x_t \notin K_l) = -\infty;
\end{equation}
\item[(ii)]  there is an additive family $\cF \subset C(E;\R)$ which separates points in $E$ such that the associated sequence $(f_\sharp \cP_\varepsilon) \in \PP(C[0,T];\R))$ is exponentially tight with speed $\beta_\varepsilon$, for every $f \in \cF$.
\end{itemize}
\end{theorem}
For a proof of this result we refer to \cite[Thm.~3]{schied1994criteria}.
\begin{remark}
Let us notice that
\begin{itemize}
\item[(a)] The exponential tightness of a family in $\PP(C[0,T];E)$ is linked via $\cF \subset C(E;\R)$ to the classical exponential tightness in $\PP(C[0,T];\R)$, for which a criterion is given by Proposition \ref{p:exp_tight_R}.
\item[(b)] If $(E,d)$ is a metric space the compact-open topology on $C([0,T];E)$ is metrizable. A possible distance is given by $d_c(x,y) = \sup_{t\in [0,T]}d(x_t,y_t)$.
\item[(c)] Condition (ii) of the above theorem can be weakened: it is enough to choose elements $f \in \cF$ such that the restriction $f|_{K_l}$ is continuous. 
\end{itemize}
\end{remark}


In order to apply Theorem \ref{t:criterium_exp_tight}, we estimate the (quadratic) energy of the particles system.
The crucial ingredient is a suitable generalization of Bernstein inequality for martingales, whose proof can be found in \cite[Lem.~2]{mariani2010large}:

\begin{lemma}\label{l.bernstein}
Let $(M_t)_{t \geq 0}$ be a continuous martingale such that $M(0) = 0$ and $\E(M_t^2) < \infty$ for every $t \geq 0$. 
If $\beta \geq 0$ and $C \in (0, +\infty)$ then for any bounded stopping time $\tau$ it holds
\begin{equation}
\P\left( \sup_{t \leq \tau} M_t > l, \, [M]_\tau \leq \beta \sup_{t \leq \tau} M_t + C \right) \leq \exp \left( -\frac{l^2}{2(\beta l + C)} \right), \qquad \quad l >0.
\end{equation}
\end{lemma}
The following proposition shows that the energy associated to the dynamic can be arbitrarily large but only with probability exponentially small.

\begin{proposition}\label{p:energy_esitmate}
Let $\P^{N,\varepsilon} \in \PP(C([0,T]; (\R^d)^N))$ be the law of a solution to \eqref{eq:state} with initial distribution $\mu_0^N$. 
If Hypotheses \ref{h1} and \ref{h2} hold, there exists a constant $C >0$ depending on $A,T$ and $\sup_{N \in \N} \int_{\R^d}|x|^2 d\mu_0^N$ such that 
\begin{equation}
\frac{\varepsilon}{N} \log \P^{N, \varepsilon} \left( \sup_{t \in [0,T]} \int_{\R^d} |x|^2 \d\mu^{N,\varepsilon}_t > l \right) \leq - \frac{(l-C)^2}{2C(l - C+ 1)}, \qquad \quad  l >0
\end{equation}
for every $\varepsilon \in (0,1)$ and $N \in \N$.
\end{proposition}

\begin{proof}
Throughout the proof we shall denote by $C >0$ a generic constant, whose value may change from line to line. 
From the It\^o formula for $\phi(x) = |x|^2$ applied to \eqref{eq:state} we get
\begin{equation}
\begin{split}
\frac{1}{N} \sum_{i = 1}^N |x_i^{N,\varepsilon}(t)|^2 &= \frac{1}{N} \sum_{i = 1}^N |x_i^N(0)|^2 + \frac{2}{N} \sum_{i=1}^N \int_0^t \fF(x_i^{N,\varepsilon}(s), \xx^{N,\varepsilon}(s)) x_i^{N,\varepsilon}(s) \d s 
+ 2\varepsilon t \\
&+ \frac{2\sqrt{\varepsilon}}{N}\sum_{i = 1}^N\int_0^t x_i^{N,\varepsilon}(s) \d W_i(s).  \\
\end{split}
\end{equation}
Further, thanks to the growth condition in Hypothesis \ref{h1}
\begin{equation*}
\begin{split}
\frac{2}{N} \sum_{i=1}^N \int_0^t & \fF(x_i^{N,\varepsilon}(s), \xx^{N,\varepsilon}(s))x_i^{N,\varepsilon}(s) \d s  \\
&\leq \frac{A}{N} \sum_{i=1}^N \int_0^t \left(1+ |x_i^{N,\varepsilon}(s)| + \frac{1}{N}\sum_{i=1}^N |x_i^{N,\varepsilon}(s)|  \right) |x_i^{N,\varepsilon}(s)| \d s \\
&\leq \frac{C}{N} \sum_{i=1}^N \int_0^t  \left[ |x_i^{N,\varepsilon}(s)|^2  + \left( \frac{1}{N} \sum_{i=1}^N  |x_i^{N,\varepsilon}(s)| \right)^2 \right] \d s \\
&\leq \frac{C}{N} \sum_{i=1}^N \int_0^t |x_i^{N,\varepsilon}(s)|^2 \d s,
\end{split}
\end{equation*}
where we used the trivial inequality $\left( \sum_{i=1}^N a_i\right)^2 \leq N \sum_{i=1}^N |a_i|^2$, $a_i \in \R$.
If we denote by $M^{N,\varepsilon}_t$ the stochastic integral $M^{N,\varepsilon}_t:= \frac{2\sqrt{\varepsilon}}{N}\sum_{i = 1}^N\int_0^t x_i^{N,\varepsilon}(s) \d W_i(s)$ we have
\begin{equation}
\frac{1}{N} \sum_{i = 1}^N |x_i^{N,\varepsilon}(t)|^2 \leq \frac{1}{N} \sum_{i = 1}^N |x_i^N(0)|^2 + \frac{C}{N} \sum_{i = 1}^N \int_0^t |x_i^{N,\varepsilon}(s)|^2 \d s + 2\varepsilon t + M^{N,\varepsilon}_t,
\end{equation}
and the classical theory for SDEs with Lipsschitz coefficients guarantees that $M^{N,\varepsilon}_t$ is a $\P$-martingale with $\E | M^{N,\varepsilon}_t |^2 < +\infty$, for every $t > 0$.
Employing the Gronwall-type inequality \eqref{gronwall} we end up with
\begin{equation}
\begin{split}
\frac{1}{N} \sum_{i = 1}^N |x_i^{N,\varepsilon}(t)|^2 &\leq \left( \frac{1}{N} \sum_{i = 1}^N |x_i^N(0)|^2 + 2\varepsilon t + M^{N,\varepsilon}_t \right) \\
&\quad + C\int_0^t \left( b + 2\varepsilon s \right)e^{C(t-s)} \d s + C\int_0^t M^{N,\varepsilon}_s e^{C(t-s)} \d s \\
&\leq C \left( 1 + \sup_{s \leq t} M^{N,\varepsilon}_s \right)
\end{split}
\end{equation}
where we used the notation $ b = \sup_{N \in \N} \int_{\R^d} |x|^2 \d \mu^N_0$. 
Hence 
\begin{equation}\label{ito^2}
\sup_{s \leq t} \int_{\R^d} |x|^2 \d\mu_s^{N,\varepsilon}  \leq C \left( 1 + \sup_{s \leq t} M^{N,\varepsilon}_s \right).
\end{equation}
Due to the independence of $\lbrace W_i \rbrace_{i= 1,\ldots, N}$, the quadratic variation of $M^{N,\varepsilon}$ can be estimated as:
\begin{equation}\label{est_quadratic_var}
\begin{split}
[M^{N,\varepsilon}]_t  &= \frac{4\varepsilon}{N^2} \sum_{i=1}^N \int_0^t |x_i^{N,\varepsilon}(s)|^2 \d s \\
&\leq \frac{C\varepsilon}{N} \sup_{s \leq t} \left(\frac{1}{N} \sum_{i=1}^N | x_i^{N,\varepsilon}(s)|^2 \right) \\
&\leq \frac{C\varepsilon}{N} \left( 1 + \sup_{s \leq t} M^{N,\varepsilon}_s \right),\qquad \forall \, t \in [0,T],
\end{split}
\end{equation}
where in the last inequality we used estimate \eqref{ito^2}. 
Summing up and employing Lemma \ref{l.bernstein} we obtain
\begin{equation}
\begin{split}
\P \left( \sup_{t \leq T} M^{N,\varepsilon}_t > m \right) &= \P \left( \sup_{t \leq T} M^{N,\varepsilon}_t > m\, , \;\;[M^{N,\varepsilon}]_T \leq \frac{C\varepsilon}{N} \left( \sup_{t \leq T} M^{N,\varepsilon}_t + 1\right) \right) \\
&\leq \exp \left( - \frac{m^2}{\frac{2C\varepsilon}{N}(m + 1)} \right), \qquad m >0,
\end{split}
\end{equation}
where the first equality follows by \eqref{est_quadratic_var}. Therefore
\begin{equation}
\frac{\varepsilon}{N}\log \P \left( \sup_{t \leq T} M^{N,\varepsilon}_t > m \right) \leq - \frac{m^2}{2C(m + 1)}, \qquad m > 0. 
\end{equation}
Employing again \eqref{ito^2} we get
\begin{equation}
\frac{\varepsilon}{N}\log \P \left( \sup_{t \leq T} \int_{\R^d} |x|^2 \d\mu^{N,\varepsilon}_t  > m +C \right) \leq - \frac{m^2}{2C(m + 1)}, \qquad m > 0. 
\end{equation}
and the result easily follows by substituting $l = m+C$.
\end{proof}

Let us notice that under Assumption \ref{h2} we have $\sup_{N \in \N} \int_{\R^d}|x|^2 \d\mu_0^N < +\infty$ and the result in Proposition \ref{p:energy_esitmate} guarantees that
\begin{equation}
\lim_{l \to +\infty} \limsup_{\substack{\varepsilon \to 0 \\ N \to \infty}} \frac{\varepsilon}{N} \log \P^{N, \varepsilon} \left( \sup_{t \in [0,T]} \int_{\R^d} |x|^2 \d\mu^{N,\varepsilon}_t > l \right) = -\infty
\end{equation}
Now we give an estimate of the continuity moduli of the dynamic: 

\begin{proposition}\label{p:mod_continuity}
Let $\P^{N,\varepsilon} \in \PP(C([0,T]; (\R^d)^N))$ be the law of a solution to \eqref{eq:state} with initial distribution $\mu_0^N$. 
Under Assumptions \ref{h1}, \ref{h2}, for any $\varphi \in C_c^{\infty}(\R^d)$  and any $\zeta > 0$ it holds
\begin{equation}
\lim_{\delta \to 0} \limsup_{\substack{\varepsilon \to 0 \\ N \to \infty}} \frac{\varepsilon}{N} \log \P^{N,\varepsilon}\left( \sup_{|t-s|<\delta} \left| \mu_t^{N,\varepsilon}(\varphi)  - \mu_s^{N,\varepsilon}(\varphi) \right| > \zeta \right) = -\infty
\end{equation}
\end{proposition}

\begin{proof}
Throughout the proof we 
we maintain the notation of Proposition \ref{p:energy_esitmate} for what concerns the constant $C = C\left(A,T, \sup_{N \in \N} \int_{\R^d}|x|^2 \d\mu^N_0\right)$.
First of all, notice that it is enough to prove the following equality
\begin{equation}
\lim_{\delta \to 0} \limsup_{\substack{\varepsilon \to 0 \\ N \to \infty}} \sup_{s \in [0,T-\delta]}\frac{\varepsilon}{N} \log \frac{T}{\delta} \P^{N,\varepsilon}\left( \sup_{t \in [s, s + \delta]} \left| \mu_t^{N,\varepsilon}(\varphi)  - \mu_s^{N,\varepsilon}(\varphi) \right| > \zeta \right) = -\infty.
\end{equation} 
A justification for this argument can be found in \cite[Thm. 7.4]{Billingsley1999convergence}.
This limit formulation is more convenient for the application of the It\^o formula: let $\varphi \in C^\infty_c(\R^d)$, then for each $s \in [0,T-\delta]$ and $t \in [s, s+\delta]$ it holds
\begin{equation}
\begin{split}
\frac{1}{N} \sum_{i = 1}^N &\varphi(x_i^{N,\varepsilon}(t)) - \frac{1}{N} \sum_{i = 1}^N \varphi(x^{N,\varepsilon}_i(s)) \\
&= \frac{1}{N} \sum_{i=1}^N \int_s^t \left( \fF(x_i^{N,\varepsilon}(r), \xx^{N,\varepsilon}(r)) \cdot  \nabla \varphi(x_i^{N,\varepsilon}(r))  + \varepsilon\tr(D_x^2\varphi(x^{N,\varepsilon}_i(r))) \right) \d r \\
&+ \frac{\sqrt{\varepsilon}}{N}\sum_{i = 1}^N\int_s^t \nabla \varphi(x_i^{N,\varepsilon}(r)) \d W_i(r)  \\
&=: A_t^{\varphi,s} + M_t^{\varphi,s}. 
\end{split}
\end{equation}
The first term can be estimated exploiting the growth condition on $F$ given in Hypothesis \ref{h1}
\begin{equation}
\begin{split}
|A_t^{\varphi,s}| &\lesssim \frac{1}{N} \sum_{i=1}^N \int_s^t \left( |\fF(x_i^{N,\varepsilon}(r), \xx^{N,\varepsilon}(r))|  + \varepsilon\right) \d r  \\
&\lesssim \delta \left( \sup_{r \in [0,T]} \frac{1}{N} \sum_{i=1}^N  |\fF(x_i^{N,\varepsilon}(r), \xx^{N,\varepsilon}(r))|  + \varepsilon\right) \\
&\lesssim \delta \left( 1+ \sup_{r \in [0,T]} \int_{\R^d} |x|^2 \d\mu^{N,\varepsilon}_r \right).
\end{split}
\end{equation}
Whence for any $\zeta > 0$
\begin{equation}
\frac{\varepsilon}{N}\log\P^{N,\varepsilon} \left( \sup_{t \in [s, s+\delta]} |A_t^{\varphi,s}| > \zeta \right) \leq \frac{\varepsilon}{N}\log\P^{N,\varepsilon} \left( \sup_{r \in [0,T]} \int_{\R^d} |x|^2 \d\mu^{N,\varepsilon}_r > \frac{\zeta}{c\delta} - 1 \right)
\end{equation}
and from Proposition \ref{p:energy_esitmate} it follows that 
\begin{equation}
\frac{\varepsilon}{N}\log\P^{N,\varepsilon} \left( \sup_{t \in [s, s+\delta]} |A_t^{\varphi,s}| > \zeta \right) \leq - \frac{\left( \frac{\zeta}{c\delta} -1 -C \right)^2}{2C\left( \frac{\zeta}{c\delta} - C \right)}.
\end{equation}
Notice that the right hand side of the above inequality goes to $-\infty$ when $\delta \downarrow 0$.

For what concerns the stochastic integral, $M^{\varphi, s}_t$ is a martingale with bounded quadratic variation:
\begin{equation}
[M^{\varphi, s}]_t \lesssim \frac{ \delta \varepsilon}{N}. 
\end{equation}
Hence, from the application of Lemma \ref{l.bernstein} (with $\beta = 0$) both to $M^{\varphi,s}$ and $-M^{\varphi,s}$ we get
\begin{equation}
\frac{\varepsilon}{N}\log\P^{N,\varepsilon} \left( \sup_{t \in [s, s+\delta]} \pm M_t^{\varphi,s} > \zeta \right) \leq - \frac{\zeta^2}{2c\delta}, 
\end{equation}
which readily implies that
\begin{equation}
\frac{\varepsilon}{N}\log\P^{N,\varepsilon} \left( \sup_{t \in [s, s+\delta]} |M_t^{\varphi,s}| > \zeta \right) \leq  \frac{\varepsilon}{N} 	\log 2 - \frac{\zeta^2}{2c\delta}.
\end{equation}
Notice that in the last passage we used the elementary inequality:
\begin{equation}\label{prob_union}
\log \P\left(  \bigcup_{i=1}^n U_i \right) \leq \log n + \bigvee_{i=1}^n \log \P(U_i),
\end{equation}
for any given probability measure and any measurable sets $U_1, \ldots, U_n$,

Summing up the estimates for $A^{\varphi,s}_t$ and  $M^{\varphi,s}_t$ we get
\begin{equation}
\begin{split}
\frac{\varepsilon}{N}&\log \frac{T}{\delta}\P^{N,\varepsilon} \left( \sup_{t \in [s,s+\delta]} \left| \mu_t^{N,\varepsilon}(\varphi)  - \mu_s^{N,\varepsilon}(\varphi) \right| > \zeta \right) \\
&= \frac{\varepsilon}{N}\log \frac{T}{\delta} + \frac{\varepsilon}{N}\log\P^{N,\varepsilon} \left( \sup_{t \in [s,s+\delta]} \left| \mu_t^{N,\varepsilon}(\varphi)  - \mu_s^{N,\varepsilon}(\varphi) \right| > \zeta \right) \\
&\leq \frac{\varepsilon}{N}\log \frac{T}{\delta} + \frac{\varepsilon}{N}\log 2  + \max \left\lbrace \frac{\varepsilon}{N} \log \P^{N,\varepsilon} \left( \sup_{t \in [s, s+\delta]} |A_t^{\varphi,s}| > \frac{\zeta}{2} \right), \right. \\
& \left.  \frac{\varepsilon}{N}\log\P^{N,\varepsilon} \left( \sup_{t \in [s, s+\delta]} |M_t^{\varphi,s}| > \frac{\zeta}{2}\right) \right\rbrace \\
&\leq \frac{\varepsilon}{N}\log \frac{T}{\delta} + \frac{\varepsilon}{N}\log 2  + \max \left\lbrace  - \frac{\left( \frac{\zeta}{c\delta} -1 -C \right)^2}{2C\left( \frac{\zeta}{c\delta} - C \right)} , \frac{\varepsilon}{N} 	\log 2 - \frac{\zeta^2}{C\delta} \right\rbrace
\end{split}
\end{equation}
Taking the limit as $\varepsilon \to 0, N \to \infty$  we get
\begin{equation}
\begin{split}
\limsup_{\substack{\varepsilon \to 0 \\ N \to \infty}} \sup_{s \in [0,T-\delta]}&\frac{\varepsilon}{N}\log \frac{T}{\delta}\P^{N,\varepsilon} \left( \sup_{t \in [s,s+\delta]} \left| \mu_t^{N,\varepsilon}(\varphi)  - \mu_s^{N,\varepsilon}(\varphi) \right| > \zeta \right)  \\
&\leq \max \left\lbrace  - \frac{\left( \frac{\zeta}{c\delta} -1 -C \right)^2}{2C\left( \frac{\zeta}{c\delta} - C \right)} , - \frac{\zeta^2}{C\delta} \right\rbrace
\end{split}
\end{equation}
Finally, when $\delta \to 0$ we get the required result.
\end{proof}

\begin{proposition}\label{p:exp_tight_currents}
Let Assumptions \ref{h1}, \ref{h2} hold. If $s_1 \in (\frac{1}{2}, 1)$ and $s_2 \in (\frac{d+2}{2}, +\infty)$, then
\begin{equation}
\lim_{l \to +\infty} \limsup_{\substack{\varepsilon \to 0,\\ N \to +\infty}} \frac{\varepsilon}{N} \log \P^{N,\varepsilon}\left( \| J^{N,\varepsilon} \|^2_{\shH^{-s}} > l \right) = -\infty
\end{equation}
\end{proposition}

\begin{proof}
From Theorem \ref{t:reg_stoch_current} we know that $J^{N,\varepsilon}$ admits a pathwise realization.  Thanks to \eqref{J_est}, \eqref{J_norm} we know that   
\begin{equation}
\| J^{N,\varepsilon} \|_{\shH^{-s}}^2 \lesssim \sum_{n,m} \int_{\R^d} (1+n^2)^{-s_1}(1+|k|^2)^{-s_2} |Z_n^m(k)|^2 \d k,
\end{equation}
where $Z_n^m(k)$ is given as in \eqref{eq:def_Z^n}:
\begin{equation}
\begin{split}
Z^m_n(k) &= \frac{1}{N}\sum_{i=1}^N\int_0^T e^m_{n,k}(t,x^{N,\varepsilon}_i) \fF(x^{N,\varepsilon}_i, \xx^{N,\varepsilon})\d t  + \frac{\sqrt{\varepsilon}}{N}\sum_{i=1}^N \int_0^T e^m_{n,k}(t,x^{N,\varepsilon}_i) \d W_i(t)\\
&+ \frac{i}{2N} \sum_{i=1}^N \sum_{j=1}^d  k_j \int_0^T  e^m_{n,k}(t,x^{N,\varepsilon}_i) \d [(x^{N,\varepsilon}_i)_j,(x^{N,\varepsilon}_i)_m]_t. 
\end{split}
\end{equation}
Relabelling the martingale part $\bar Z^m_{n,T}(k):= \frac{\sqrt{\varepsilon}}{N}\sum_{i=1}^N \int_0^T e^m_{n,k}(t,x^{N,\varepsilon}_i)\d W^{N,\varepsilon}_i(t)$ (we emphasize the dependence on the final time) and using the growth conditions on $\fF$ we get
\begin{equation}
\begin{split}
|Z^m_{n,T}(k)|^2 &\leq \|e^m_{n,k} \|^2_{\infty} \frac{C}{N}\sum_{i=1}^N \left( \int_0^T |x^{N,\varepsilon}_i |^2 \d t + CT|k|^2  \right) +  |\bar Z^m_{n,T}(k)|^2 \\
&\leq C  \left( \sup_{t \in [0,T]} \int_{\R^d} |x|^2 \d \mu^{N, \varepsilon}_t + |k|^2 \right) + |\bar Z^m_{n,T}(k)|^2, \qquad \text{ for some } C > 0.
\end{split}
\end{equation}
Now set  $\gamma:= \sum_{m,n} \int_{\R^d} (1+n^2)^{-s_1}(1+|k|^2)^{-s_2}\d k$ and denote by $\DD$ the counting measure on $\lbrace 1, \ldots, d \rbrace \times \N$.
Let us introduce the probability measure $\Gamma$ on $\lbrace 1, \ldots, d \rbrace \times \N \times \R^d$ by the following
\begin{equation}
\Gamma(\d a) := \gamma^{-1} (1+n^2)^{-s_1}(1+|k|^2)^{-s_2} d \DD(m,n)\d k, \qquad a= (m,n,k). 
\end{equation}
Then we have
\begin{equation}\label{first_est_J}
\begin{split}
\| J^{N,\varepsilon} \|_{\shH^{-s}}^2 &\lesssim \gamma \int |Z_T(a)|^2 \Gamma(\d a) \\
&\lesssim \gamma \left(1+  \sup_{t \in [0,T]} \int_{\R^d} |x|^2 \d \mu^{N, \varepsilon}_t \right) + \gamma \int |\bar Z_T(a)|^2 \Gamma(\d a),
\end{split}
\end{equation}
where we used the fact that $\int \left( 1+ |k|^2 \right) \Gamma(\d a) < +\infty$. Thanks to Proposition \ref{p:energy_esitmate} we know that 
\begin{equation}\label{eq:energy_limit}
\lim_{l \to + \infty} \limsup_{\substack{\varepsilon \to 0,\\ N \to +\infty}}\frac{\varepsilon}{N} \log \P^{N, \varepsilon} \left( 1+ \sup_{t \in [0,T]} \int_{\R^d} |x|^2 \d\mu^{N,\varepsilon}_t > l \right) = - \infty.
\end{equation}

Let us now concentrate on the second term on the right hand side of \eqref{first_est_J}. If we define $\bar Z_t(a):= \bar Z^{m}_{n,t}(k)$, with $t \in [0,T]$ then $\bar Z_t(a)$ and $\bar Y_t(a):= \bar Z_t(a)^2 - [\bar Z(a)]_t$ are continuous martingales. Hence 
\begin{equation}
\begin{split}
\int |\bar Z_T(a)|^2 \Gamma(\d a) &= \int \bar Y_T(a) \Gamma(\d a) + \int [\bar Z(a)]_T \Gamma(\d a) \\
&=:\bar X_T + \int [\bar Z(a)]_T \Gamma(\d a).
\end{split}
\end{equation}
We start by estimating the bracket between $\bar Z(a)$ and $\bar Z(b)$:
\begin{equation}
\begin{split}
[\bar Z(a), \bar Z(b)]_t  &\lesssim \frac{\varepsilon}{N^2} \sum_{i=1}^N \int_0^t e_a(s, x_i^{N,\varepsilon}) \cdot  e_b(s, x_i^{N,\varepsilon}) \d s \\
&\lesssim \frac{\varepsilon}{N} \| e_a\|_{\infty} \| e_b\|_{\infty} t \lesssim \frac{ \varepsilon}{N},
\end{split}
\end{equation}
so that 
\begin{equation}\label{eq:Z_dgamma}
\gamma \int |\bar Z_T(a)|^2 \Gamma(\d a) \leq \gamma \bar X_T + \frac{C \varepsilon}{N}, \qquad \text{ for some } C > 0.
\end{equation}
Concerning the bracket between $\bar Y(a)$ and $\bar Y(b)$ we have
\begin{equation}
\begin{split}
[\bar Y(a),\bar Y(b)]_t &= 4 \int_0^t \bar Z_s(a)\bar Z_s(b) \d [\bar Z(a),\bar Z(b)]_s \\
&\lesssim \frac{\varepsilon}{N} \int_0^t \left( \bar Z_s(a)^2 + \bar Z_s(b)^2 \right) \d s \\
&\lesssim \frac{\varepsilon}{N} \int_0^t \left( \bar Y_s(a) + \bar Y_s(b) \right) \d s + \frac{ \varepsilon^2}{N^2}.
\end{split}
\end{equation}
Let us now observe that the process $\bar X_t := \int \bar Y_t(a) \Gamma(\d a)$ is itself a martingale with quadratic variation given by
\begin{equation}
\begin{split}
[\bar X]_t &= \int [\bar Y(a), \bar Y(b)]_t \Gamma(\d a)\Gamma(\d b) \\
&\lesssim \frac{\varepsilon}{N} \int_0^t \bar X_s \d s + \frac{ \varepsilon^2}{N^2}\\
&\lesssim \frac{\varepsilon}{N} \sup_{s \in [0,t]} \bar X_s + \frac{\varepsilon^2}{N^2}.
\end{split}
\end{equation}
Hence, from Lemma \ref{l.bernstein} and the computations above there exists $C >0$ such that
\begin{equation}
\begin{split}
\P^{N,\varepsilon} \left( \bar X_T > l\right) &= \P^{N,\varepsilon} \left( \bar X_T > l, \;\bar X_T \leq \frac{C\varepsilon}{N} \sup_{t \in [0,T]} \bar X_t + \frac{C \varepsilon^2}{N^2}\right) \\ 
&\leq \exp\left( -\frac{l^2}{C\frac{\varepsilon}{N}\left( l + \frac{\varepsilon}{N} \right)} \right).
\end{split}
\end{equation}
So that recalling \eqref{eq:Z_dgamma} we have
\begin{equation}
\lim_{l \to + \infty} \limsup_{\substack{\varepsilon \to 0,\\ N \to +\infty}}\frac{\varepsilon}{N}\log \P^{N,\varepsilon} \left( \int |\bar Z_T(a)|^2 \Gamma(\d a) >l \right) = -\infty,
\end{equation}
which together with \eqref{eq:energy_limit} gives the required result.
\end{proof}

Now we can state the main result of this section, whose proof is based on the application of Theorem \ref{t:criterium_exp_tight}. 

\begin{theorem}\label{t:exp_tight}
The family of probability measures $\P^{N,\varepsilon} \circ (\mu^N,J^{N,\varepsilon})^{-1} \in \PP(\cX)$ is exponentially tight with speed $\varepsilon / N$. 
\end{theorem}

\begin{proof}
We split the proof in two parts, showing the result for $\P^{N,\varepsilon} \circ (\mu^N)^{-1}$ and $\P^{N,\varepsilon} \circ (J^{N,\varepsilon})^{-1}$ separately. 

Let us start with the family $\P^{N,\varepsilon} \circ (J^{N,\varepsilon})^{-1}$.
For every $l >0$, define $K_l := \lbrace J \in \hH^{-s}: \| J \|^2_{\shH^{-s}} \leq l \rbrace$, which is compact with respect to the weak* topology.
Choosing $s_1 \in (\frac{1}{2}, 1)$, $s_2 \in (\frac{d+2}{2}, +\infty)$, the exponential tightness directly follows by Proposition \ref{p:exp_tight_currents}.

For what concerns the empirical measure we show that conditions (i)-(ii) of Theorem \ref{t:criterium_exp_tight} are satisfied.
\newline
(i)  For every $l > 0$ introduce the set $K_l \subset \PP_1(\R^d)$:
\begin{equation}
K_l:= \lbrace \mu \in \PP_1(\R^d): \int_{\R^d} |x|^2 \d\mu(x) \leq l \rbrace.
\end{equation}
 Thanks to Prokhorov theorem (see also \eqref{UI}), $K_l$ is relatively compact in $\PP_1(\R^d)$ endowed with the narrow topology, for every $l >0$.
The application of Proposition \ref{p:energy_esitmate} readily implies that 
\begin{equation}
\lim_{l \to +\infty} \limsup_{\substack{\varepsilon \to 0 \\ N \to \infty}} \frac{\varepsilon}{N} \log \P^{N, \varepsilon} \left( \exists t \in [0,T]: \mu^{N}_t \notin K_l  \right) = -\infty.
\end{equation}
\newline
(ii) For every $f \in C_c(\R^d)$ define the map $\hat f \in C(\PP_1(\R^d);\R)$ by
\begin{equation}
\hat f(\mu):= \int_{\R^d} f(x)\d\mu(x), \qquad \mu \in \PP_1(\R^d).
\end{equation}
From \cite[Thm.~3.4.4]{ethier1986markov} and Section \ref{sec.meas-theory}, we know that $\cF:= \lbrace \hat f : f \in C_c(\R^d) \rbrace$ separates points in  $\PP_1(\R^d)$ and it is an additive family, i.e. $C_c(\R^d)$ is closed under addition.
Therefore, from Theorem \ref{t:criterium_exp_tight} it is enough to study exponential tightness of the family $\P^{N,\varepsilon} \circ (\langle \mu^{N,\varepsilon},\hat f \rangle )^{-1} \in \PP(C[0,T];\R)$, where we tacitly assume that $\hat f(\mu)(t) = \hat f(\mu_t)$, whenever $\mu \in C([0,T]; \PP_1(\R^d))$.
Let us also notice that the same argument goes through just considering smooth functions $f \in C^\infty_c(\R^d)$, whose linear envelope is uniformly dense in $C_c(\R^d)$.   

Let us now apply Theorem \ref{p:exp_tight_R} to the family $\P^{N,\varepsilon} \circ (\langle \mu^{N,\varepsilon},\hat f \rangle )^{-1} \in \PP(C[0,T];\R)$.
Condition (a) in Theorem \ref{p:exp_tight_R} is satisfied thanks to the continuity of the map $\hat f: C([0,T];\PP_1(\R^d)) \to C([0,T];\R))$.
For what concerns point (b), let us fix a compact set $K \subset \PP_1(\R^d)$ and  use \cite[Lem.~3.2]{jakubowski1986skorokhod} to select a countable additive family $\cF_K \subset \cF$ which separates points in $K$.   
From the application of Proposition \ref{p:mod_continuity} and \cite[Lem.~3.3]{jakubowski1986skorokhod} we finally conclude the proof. 
\end{proof}

\subsection{Goodness of the rate functional}

Here we present a direct proof of the goodness of $I: \cX \to [0,+\infty]$ under less restrictive assumptions than Hypothesis \ref{h.lip}, for which a lower bound estimate holds (see Section 5).
Recall indeed that when a family of probability measures is exponentially tight and satisfies a large deviation lower bound, then the associated rate functional  is automatically good.

\begin{lemma}\label{l:rate_h}
Let $(\mu,J) \in \cX$ with  $I(\mu,J) < \infty$.
Then there exists $h \in L^2(Q, \tilde \mu; \Rd)$ such that 
\begin{equation}
I(\mu,J) = \frac{1}{2} \int_0^T \langle \mu(s), |h(t,\cdot)|^2 \rangle \d s
\end{equation}
and the pair $(\mu,J)$ satisfies 
\begin{equation}
\begin{system}
\partial_t \mu_t + \dive J_t = 0, \quad \mu|_{t=0} = \mu_0,\\
J_t = \left( F(\cdot, \mu_t) + h \right)\mu_t
\end{system}
\end{equation}
in the distributional formulation.
\end{lemma}

\begin{proof}
The variational formulation of the rate functional \eqref{rate_variational} and the linearity of the map $\eta \mapsto J(\eta)$ assures that for any $c \in \R$
\begin{equation}\label{eq:I_c}
cJ(\eta)  - c\int_0^T \langle \mu_t, \eta(t, \cdot) \cdot \fF(\cdot, \mu_t) \rangle \d t - \frac{c^2}{2}\int_0^T \langle \mu_t, \left| \eta(t,\cdot)  \right|^2 \rangle \d t \leq  I(\mu,J), \qquad \forall \, \eta \in C^{\infty}_c(Q;\R^d) 
\end{equation}
The maximum on the left-hand side is reached in correspondence with
\[c = \frac{\left| J(\eta) - \int_0^T \langle \mu_t, \eta(t, \cdot) \cdot \fF(\cdot, \mu_t) \rangle \d t \right|^2}{\int_0^T \langle \mu_t, \left| \eta(t,\cdot)  \right|^2 \rangle \d t},\] 
substituting it in \eqref{eq:I_c} we get 
\begin{equation}
\left| J(\eta) - \int_0^T \langle \mu_t, \eta(t, \cdot) \cdot \fF(\cdot, \mu_t) \rangle \d t \right|^2 \leq 2I(\mu,J) \int_0^T \langle \mu_t, \left| \eta(t,\cdot)  \right|^2 \rangle \d t,
\end{equation}
whence $\eta \mapsto J(\eta) - \int_0^T \langle \mu_t, \eta(t, \cdot) \cdot \fF(\cdot, \mu_t) \rangle \d t$ is a bounded linear functional on $C^{\infty}_c(Q,\tilde \mu;\Rd)$ which can be extended to $L^2( Q,\tilde \mu; \Rd)$.
Thanks to the Riesz representation theorem there exists $\hh \in L^2( Q,\tilde \mu; \Rd)$ such that 
\begin{equation}
J(\eta) - \int_0^T \langle \mu_t, \eta(t, \cdot) \cdot \fF(\cdot, \mu_t) \rangle \d t  = \int_0^T \langle \mu_t, \eta(t, \cdot) \cdot \hh(t, \cdot) \rangle \d t. 
\end{equation}
This implies that
\begin{equation}
J(\eta) = \int_0^T \langle \mu_t, \eta(t, \cdot) \cdot \left( \fF(\cdot, \mu_t) + \hh(t, \cdot) \right) \rangle dt  
\end{equation}
and $(\mu,J)$ satisfies $\partial_t \mu_t  + \dive \left( \fF(\cdot, \mu_t)\mu_t + \hh(t,\cdot)\mu_t \right)=0$ in a distributional sense \eqref{ec_v}.
Given such a solution $(\mu,J)$ we also get
\begin{equation}
\begin{split}
I(\mu,J)  &= \sup_{\eta \in C^\infty_c(Q;\Rd)} \Big\lbrace \int_0^T \langle \mu_t, \eta(t, \cdot) \cdot \hh(t, \cdot) \rangle \d t - \frac{1}{2}\int_0^T \langle \mu_t, \left| \eta(t,\cdot)  \right|^2 \rangle \d t \Big\rbrace\\
&= \frac{1}{2}\int_0^T \langle \mu_t, \left| \hh(t,\cdot)  \right|^2 \rangle \d t - \inf_{\eta \in C^\infty_c(Q;\Rd)} \frac{1}{2} \int_0^T \langle \mu_t, \left| \hh(t, \cdot) - \eta(t,\cdot)  \right|^2 \rangle \d t \\
&= \frac{1}{2}\int_0^T \langle \mu_t, \left| \hh(t,\cdot)  \right|^2 \rangle \d t,
\end{split}
\end{equation}
which is the required result.
\end{proof}

\begin{lemma}\label{goodness}
Given $l \in \R_+$ define $\cX_l:= \lbrace (\mu,J) \in \cX: \int_{\R^d}|x|^2 \d\mu_0(x) \leq l \rbrace \subset \cX$.  
Then, for any $l \in \R_+$ the rate functional $I: \cX_l \to [0,+\infty]$ is good, i.e. has compact sublevels.
\end{lemma}

\begin{proof}
Let $\cA := \lbrace (\mu,J) \in \cX: I(\mu,J) \leq a < +\infty\rbrace$. 
If $(\mu,J) \in \cA$ we know from Lemma \ref{l:rate_h} that there exists $\hh \in L^2(Q, \tilde \mu; \Rd)$ such that $\partial_t \mu_t + \dive \left( \left( \fF(\cdot, \mu_t) + \hh(t,\cdot) \right)\mu_t \right) = 0$ and 
\begin{equation}
I(\mu,\hh) = \frac{1}{2}\int_0^T \langle \mu_t, \left| \hh(t,\cdot)  \right|^2 \rangle \d t.
\end{equation}
Fix $ l \in \R_+$ and take a sequence $(\mu^n,J^n= \hh^n\mu^n) \in \cA \cap \cX_l$.
From \cite[Prop.~5.3]{fornasier2018mean} there exists a constant $\tilde C$, depending on $C, T, a$ and $\int_{\R^d}|x|d\mu_0^n(x)$ such that 
\begin{equation}\label{est:moments}
\sup_{t \in [0,T]} \int_{\R^d} |x|^2 \d\mu^n_t(x) \leq \tilde C \left( 1+ \int_{\R^d} |x|^2  \d\mu_0^n(x)\right) \leq \tilde C (1+l).
\end{equation}
To get the equicontinuity property, let us follow the strategy of the proof of Theorem \ref{ex_uniq_vlasov} in Appendix (part (i), step 3).
In particular, for every $n \in \N$ we define $(\mu^{n,k})_{k \in\N}$ as in \eqref{empirical_k} (using the characteristic equations associated to $\vv^n$) and from \eqref{est:moments}, \eqref{theta_estimate} and \eqref{conv_wass_muk} we get 
\begin{equation}\label{prop_mu_nk}
\sup_{n,k \in \N} \sup_{t\in [0,T]} \int_\Rd |x|^2 \d \mu_t^{n,k}(x) < +\infty, \qquad \lim_{k \to +\infty} \sup_{t \in [0,T]} W_1(\mu^{n,k}_t, \mu^n_t) = 0
\end{equation}
Now, for the empirical measures $(\mu^{n,k})_{k \in \N}$ the equicontinuity follows from the same computation as in \eqref{equicontinuity_k}:
if $s \leq t \in [0,T]$
\begin{equation}
\begin{split}
W_1(\mu^{n,k}_s, \mu^{n,k}_t) \leq |t-s| \int_\Rd |x| \d \mu_t^{n,k}(x) \leq l |t-s|, 
\end{split}
\end{equation} 
where $l \in \R_+$ does not depend neither on $k$ nor on $n$, thanks to \eqref{prop_mu_nk}.
Employing \cite[Prop.~7.1.3]{ambrosio2008gradient} we have the finally get
\begin{equation}\label{equicontinuity_n}
W_1(\mu^{n}_s, \mu^{n}_t) \leq \liminf_{k \to +\infty} W_1(\mu^{n,k}_s, \mu^{n,k}_t) \leq l |t-s|.
\end{equation}
Thanks to estimate \eqref{est:moments} and \eqref{equicontinuity_n}, Ascoli-Arzel\`a theorem provides the required compactness.

For what concerns the sequence $J^n$, let us denote $\vv^n(t,x):= \fF(x,\mu^n_t) + \hh^n(t,x)$ so that
\begin{equation}
J^n(\phi) = \int_0^T \langle \mu^n_t, \phi(t, \cdot) \cdot \vv^n(t,\cdot) \rangle \d t, \qquad  \forall \,\phi \in C_c^\infty(Q;\R^d),\; \forall \, n \in \N,  
\end{equation}
Hypothesis \ref{h1} and estimate \eqref{est:moments} guarantee that 
 \begin{equation}\label{uniform_bound_vn}
 \sup_{n \in \N}\int_0^T \int_{\R^d}|\vv^n(t,x)|^2 \d \mu^n_t(x) \d t< +\infty.
 \end{equation}

Using the same strategy as in the proof of \cite[Thm~5.4.4]{ambrosio2008gradient} we can deduce the existence of a map $\vv: [0,T] \times \R^d \to \R^d$, such that $\vv \in L^2(Q, \tilde \mu; \R^d)$ and 
\begin{equation}
\lim_{n \to +\infty} \int_0^T \langle \mu_t^n, \phi(t,\cdot) \cdot \vv^n(t,\cdot) \rangle \d t  = \int_0^T \langle \mu_t, \phi(t,\cdot) \cdot \vv(t,\cdot)\rangle \d t,
\end{equation}
for every $\phi \in C_c^\infty(Q; \R^d)$. 
Moreover, the continuity of the map $F$ with respect to the Wasserstein distance implies that  $\vv(t,x) := \fF(x, \mu_t) + \hh(t, x)$. 
This guarantees the weak* convergence (against test function $\phi \in C_c^\infty(Q; \R^d)$) of the $\R^d$-valued measures $J^n:= \vv^n\mu^n$ towards the limit $J:= \vv\mu$.

Let us now fix $\phi \in \hH^s$. 
By density there exists a sequence $\phi^k \in C_c^\infty(Q;\R^d)$ such that $\phi^k \to \phi$ in $\hH^s$ and
\begin{equation}
J^n(\phi) = \int_0^T \langle \mu_t^n, \left( \phi(t,\cdot) - \phi^k(t,\cdot)\right) \cdot \vv^n(t,\cdot) \rangle \d t + \int_0^T \langle \mu_t^n, \phi^k(t,\cdot) \cdot \vv^n(t,\cdot) \rangle \d t. 
\end{equation}
To pass to the limit as $k \uparrow +\infty$ observe that $\hH^s \hookrightarrow C^0(Q)$, so that
\begin{equation}
\begin{split}
\int_0^T &\langle \mu_t^n, \left| \phi(t,\cdot) - \phi^k(t,\cdot)\right|  \left| \vv^n(t,\cdot) \right| \rangle \d t \\
&\leq  \left(\int_0^T \langle \mu_t^n, \left| \phi(t,\cdot) - \phi^k(t,\cdot)\right|^2 \rangle \d t \right)^{1/2}\left( \int_0^T \langle \mu_t^n,  \left| \vv^n(t,\cdot) \right|^2 \rangle \d t\right)^{1/2} \\
&\leq c \| \phi - \phi^k \|_{C^0(Q)} \longrightarrow 0, \qquad \text{ as } k \uparrow + \infty,
\end{split}
\end{equation}
where in the last inequality we used the uniform bound \eqref{uniform_bound_vn}. 
As a consequence we have $J^n(\phi) \to J(\phi)$ for every $\phi \in \hH^s$, and we conclude.
  \end{proof}

\section{Large deviations lower bound}

This section is devoted to the proof of the large deviation lower bound. 
The proof is divided in two main parts: a first analytical step exploiting the deterministic recovery sequence obtained in \cite{fornasier2018mean} and a second probabilistic part in which a suitable tilt of the law associated to \eqref{eq:state} is taken into account.   

Notice that the construction of the recovery sequence in \cite{fornasier2018mean} requires the initial data to have uniformly compact support. This motivates the introduction of Hypothesis \ref{h.comp} hereinafter.
For what concerns the velocity field $F$, we will assume Hypothesis \ref{h.lip}, which guarantees the uniqueness of solutions to \eqref{eq:state} and of  the Vlasov equation 
\begin{equation}
\partial_t \mu_t + \dive \left( (F(x,\mu_t)) \mu_t \right) = 0,
\end{equation}
as it is stated in Theorem \ref{ex_uniq_vlasov} in Appendix.


To derive the large deviation lower bound we profit by the link with the $\Gamma$-convergence of the relative entropy functional (we refer to e.g.\cite{mariani2012gamma} for  a detailed analysis on this topic).
The following general theorem has been proved in \cite[Thm.~3.4]{mariani2012gamma} in a Polish space setting, but the proof also applies to the setting of a completely regular topological space.  

\begin{theorem}\label{t.lb_entropia}
Let $\cP_\varepsilon$ be a family of probability measures on a completely regular topological space $X$ and $\lbrace \beta_\varepsilon \rbrace_\varepsilon$ such that $\lim_{\varepsilon \downarrow 0} \beta_\varepsilon = 0$.
Let also $I: X \to[0,+\infty]$ be a lower semicontinuous functional. 
Then the following are equivalent
\begin{itemize}
\item[(a)] $\cP_\varepsilon$ satisfies a large deviations lower bound  with speed $\beta_\varepsilon$ and rate functional $I$;
\item[(b)] For any point $x \in X$ there exists a sequence $\cQ_\varepsilon^x \in \PP(X)$ weakly converging to the Dirac measure  $\delta_x$ such that  
\begin{equation}
\limsup_{\varepsilon \downarrow 0} \beta_\varepsilon \ent(\cQ_\varepsilon^x | \cP_\varepsilon) \leq I(x)
\end{equation}
\end{itemize}  
\end{theorem}


The application of the above result relies on the construction of a suitable recovery sequence for which the entropy remains bounded.
Such a sequence can be obtained starting from the deterministic guess provided in \cite[Thm.~3.2]{fornasier2018mean} which we briefly report here:

\begin{theorem}\label{t:recovery_det}
Let $(\mu,\hh) \in AC([0,T]; \PP_1(\R^d)) \times L^2( Q,\tilde \mu; \Rd)$ be a distributional solution to $\partial_t \mu + \nabla \cdot J = 0$, with 
$J = \hh \tilde \mu \ll \tilde \mu$, and let $\mu^N_0 \in \PP^N(\R^d)$ be a sequence of initial measures with uniformly compact support such that  $W_1(\mu^N_0, \mu_0) \to 0$ as $N \uparrow +\infty$.
Then there exists a sequence $(\yy^N,\hh^N) \in AC([0,T]; (\R^d)^N) \times L^1([0,T];(\R^d)^N)$ such that 
\begin{itemize}
\item[(a)] $\sigma[\yy_0^N] = \mu_0^N$, for every $N \in \N$;
\item[(b)] $\hh^N(t,x) := \hh(t,x) + \fF(x,\sigma^N_t) - \fF(x,\mu_t)$ for every $x \in \R^d$, $t \in [0,T]$; 
\item[(c)] the equation
\begin{equation}
\label{eq:recovery}
\frac{\d}{\d t} y_i^{N}(t)=\fF(y_i^{N}(t),\yy^{N}(t)) + \hh^N(t,y_i^N(t)) ,\quad i=1,\cdots, N,\qquad
  \yy^N(0)=\xx^N_0,
\end{equation} 
holds for every $N \in \N$; 
\item[(d)] The sequence $\sigma^N:= \sigma[\yy^N]$ satisfies 
\begin{equation}\label{eq:control_measures_n}
 \sup_{N \in \N} \sup_{t \in [0,T]} \int_{\R^d} |y|^2 \d \sigma^N(y) < +\infty,
\end{equation}
moreover $\sigma^N \to \mu$ in $C([0,T];\PP_1(\R^d))$ and $\hh^N\sigma^N \weakto^* \hh\mu$ in $\MM(Q; \R^d)$;
\item[(e)] if we denote by $I:AC([0,T]; \PP_1(\R^d)) \times \MM(Q; \R^d)$ the functional
\begin{equation}
I(\pi, \nnu) := 
\begin{system}
I(\pi, \gg):= \int_0^T \langle \pi_t, |g(t,\cdot)|^2 \rangle \d t,\qquad \text{ if } \d \nnu = \gg \d \pi,\; g \in L^2( Q,\tilde \pi; \Rd) \\
+ \infty \qquad \qquad \qquad \qquad  \qquad \qquad \;\; \text{ otherwise},
\end{system}
\end{equation}
then the following inequality holds
\begin{equation}
\limsup_{N \to \infty} I(\sigma^N,\hh^N) \leq I(\mu,\hh).
\end{equation}  
\end{itemize}
\end{theorem}


Now, given the controls $(\yy^N, \hh^N) \in AC([0,T]; (\R^d)^N) \times L^1([0,T];(\R^d)^N)$, we force equation \eqref{eq:state} to satisfy the following perturbed dynamic:
\begin{equation}
  \label{eq:tilted}
\d x_i^{N,\varepsilon,h}(t)=\fF^N(x_i^{N,\varepsilon,h}(t),\xx^{N,\varepsilon,h}(t))\d t + \hh^N(t,y_i^N(t)) \d t+ \sqrt{\varepsilon}\d W_i(t) ,\quad
  \xx^N(0)=\xx^N_0,
\end{equation} 
where $i=1,\ldots, N$.
Under Assumption \ref{h.lip} there exists a unique strong solution  $x_i^{N,\varepsilon,h}$ to \eqref{eq:tilted} with associated empirical measures
\begin{equation}
\mu^{N,\varepsilon,h}_t := \frac{1}{N} \sum_{i =1}^N \delta_{x_i^{N,\varepsilon,h}(t)}. 
\end{equation}

We can prove the following

\begin{proposition}\label{p:wasserstein_esitmate}
Let Hypotheses \ref{h2} and \ref{h.lip} be in force and denote by $\P^{N,\varepsilon}_h \in \PP(C([0,T]; (\R^d)^N))$ the law of the solution to equation \ref{eq:tilted} with initial distribution $\mu_0^N$. 
Then for every $\delta > 0$
\begin{equation}
\lim_{\substack{\varepsilon \to 0\\ N \to +\infty}} \P^{N, \varepsilon}_h \left( \sup_{t \in [0,T]} \frac{1}{N}\sum_{i=1}^N \big|x_i^{N,\varepsilon,h}(t) - y^N_i(t)\big|^2 > \delta \right) =0,
\end{equation}
which in particular yields
\begin{equation}
\lim_{\substack{\varepsilon \to 0\\ N \to +\infty}} \P^{N, \varepsilon}_h \left( \sup_{t \in [0,T]} W^2_2(\mu^{N,\varepsilon,h}_t, \sigma^{N}_t) > \delta \right) =0.
\end{equation}
\end{proposition}

\begin{proof}
Take the difference between equations \eqref{eq:tilted} and \eqref{eq:recovery}:
\begin{equation}
x_i^{N,\varepsilon,h}(t) - y^N_i(t) = \int_0^t \left( \fF^N(x_i^{N,\varepsilon,h}(s),\xx^{N,\varepsilon,h}(s)) - \fF^N(y_i^{N}(s),\yy^{N}(s)) \right) \d s + \sqrt{\varepsilon} \d W_i(t).
\end{equation}
From the It\^o formula for $\phi(x) = |x|^2$ we get
\begin{equation}
\begin{split}
\big|x_i^{N,\varepsilon,h}(t) - y^N_i(t)\big|^2 &= 2\int_0^t \left( \fF^N(x_i^{N,\varepsilon,h}(s),\xx^{N,\varepsilon,h}(s)) - \fF^N(y_i^{N}(s),\yy^{N}(s)) \right)\cdot\left( x_i^{N,\varepsilon,h}(s) - y^N_i(s) \right) \d s \\
&+ 2\sqrt{\varepsilon} \int_0^t \left( x_i^{N,\varepsilon,h}(s) - y^N_i(s) \right) \d W_i(s) + 2\varepsilon t \\
&\lesssim \int_0^t \left(\left| x_i^{N,\varepsilon,h}(s) - y_i^{N}(s)\right| + \left|W_1(\mu^{N,\varepsilon,h}_s, \sigma^{N}_s) \right| \right)\left| x_i^{N,\varepsilon,h}(s) - y^N_i(s) \right| \d s \\
&+ \sqrt{\varepsilon} \int_0^t \left( x_i^{N,\varepsilon,h}(s) - y^N_i(s) \right) \d W_i(s) + \varepsilon t.\\
 \end{split}
\end{equation}
Employing inequality $W_1(\mu^{N,\varepsilon,h}_s, \sigma^{N}_s) \leq \frac{1}{N}\sum_{i=1}^N |x_i^{N,\varepsilon,h}(s) - y^N_i(s)|$ and averaging in $N$ we have
\begin{equation}
\begin{split}
\frac{1}{N} \sum_{i=1}^N \big| x_i^{N,\varepsilon,h}(t) - y^N_i(t)\big|^2 &\lesssim \frac{1}{N} \sum_{i=1}^N \int_0^t \left| x_i^{N,\varepsilon,h}(s) - y_i^{N}(s)\right|^2 \d s +  \int_0^t \left( \frac{1}{N} \sum_{j=1}^N\left| x_j^{N,\varepsilon,h}(s) - y^N_j(s) \right| \right)^2 \d s \\
&+ \frac{\sqrt{\varepsilon}}{N} \sum_{i=1}^N  \int_0^t \left( x_i^{N,\varepsilon,h}(s) - y^N_i(s) \right) \d W_i(s) + \varepsilon t\\
&\lesssim \frac{1}{N} \sum_{i=1}^N \int_0^t \left| x_i^{N,\varepsilon,h}(s) - y_i^{N}(s)\right|^2 \d s + M^{N,\varepsilon}_h(t) + \varepsilon t,
\end{split}
\end{equation}
where we shorthand $M^{N,\varepsilon}_h(t) := \frac{\sqrt{\varepsilon}}{N} \sum_{i=1}^N  \int_0^t ( x_i^{N,\varepsilon,h}(s) - y^N_i(s) ) \d W_i(s)$ for the martingale part. 
Using the Gronwall inequality \eqref{gronwall} we firstly get
\begin{equation}
\begin{split}
\frac{1}{N} \sum_{i=1}^N \big| x_i^{N,\varepsilon,h}(t) - y^N_i(t)\big|^2 &\lesssim \varepsilon t + M^{N,\varepsilon}_h(t) + \int_0^t \left( \varepsilon s + M^{N,\varepsilon}_h(s) \right) e^{C(t-s)}\d s\\ 
&\lesssim \varepsilon + \sup_{s\le t} M^{N,\varepsilon}_h(s), 
\end{split}
\end{equation}
from which
\begin{equation}
\sup_{s \leq t} W^2_2(\mu^{N,\varepsilon,h}_t, \sigma^{N}_t) \leq C \left( \varepsilon + \sup_{s\le t} M^{N,\varepsilon}_h(s) \right),
\end{equation}
for some constant $C \ge 0$.  
Employing now Lemma \ref{l.bernstein} to control the martingale term $M^{N,\varepsilon}_h(t)$, if we proceed as in Proposition \ref{p:energy_esitmate} (notice that the initial data cancel out) we easily get that
\begin{equation}
\P^{N,\varepsilon}_h \left( \sup_{s \leq t} \frac{1}{N}\sum_{i=1}^N \big|x_i^{N,\varepsilon,h}(t) - y^N_i(t)\big|^2 > \delta \right) \longrightarrow 0, \qquad \quad  \text{ if } N\uparrow +\infty, \varepsilon \downarrow 0,
\end{equation}
which is the required estimate, due to the arbitrariness of $\delta > 0$.
\end{proof}

The corresponding result for the associated currents is contained in the following

\begin{proposition}\label{p:diff_current_esitmate}
Let $\P^{N,\varepsilon}_h \in \PP(C([0,T]; (\R^d)^N))$ be the law of the solution to equation \eqref{eq:tilted} with initial distribution $\mu_0^N$. 
If Hypotheses \ref{h2} and \ref{h.lip} hold, then for every $\delta > 0$
\begin{equation}
\lim_{\substack{\varepsilon \to 0\\ N \to +\infty}} \P^{N, \varepsilon}_h \left( \left| J^{N,\varepsilon,h}(\eta) - \nu^N(\eta) \right|^2 > \delta \right) =0, \qquad \forall \; \eta \in C_c^{\infty}(Q; \R^d).
\end{equation}
\end{proposition}
\begin{proof}
Let us start by writing the two currents:
\begin{equation}
\begin{split}
\jJ^{N,\varepsilon,h}(\eta) 
&= \frac{1}{N} \sum_{i=1}^N \int_0^T \eta(t, x_i^{N,\varepsilon,h}(t)) \cdot  \left( \fF(x_i^{N,\varepsilon,h}(t), \xx^{N,\varepsilon,h}(t))  + \hh^N(t, y_i^{N}(t)) \right) \d t \\
&+ \frac{\sqrt{\varepsilon}}{N} \sum_{i=1}^N  \int_0^T \eta(t, x_i^{N,\varepsilon,h}(t)) \d W_i(t) + \frac{\varepsilon}{N} \sum_{i=1}^N\int_0^T \dive (\eta)(t, x_i^{N,\varepsilon,h}(t)) \d t \\
&\\
\nnu^N(\eta) &= \frac{1}{N} \sum_{i=1}^N \int_0^T \eta(t, y_i^{N}(t)) \cdot \left( \fF(y_i^{N}(t), \yy^{N}(t))  + \hh^N(t, y_i^{N}(t))\right) \d t, \\
\end{split}
\end{equation}
where we can assume $\eta \in C^\infty_c(Q;\R^d)$ (henceforth we also employ the Lipschitz character of the map $x \mapsto \eta(t,x)$. 
Hence
\begin{equation}
\begin{split}
&\left| \jJ^{N,\varepsilon,h}(\eta) -  \nnu^N(\eta) \right|^2 \\
&\quad \lesssim  \left| \frac{1}{N} \sum_{i=1}^N \int_0^T \left( \eta(t, x_i^{N,\varepsilon,h}(t)) \cdot \fF(x_i^{N,\varepsilon,h}(t), \xx^{N,\varepsilon,h}(t))  - \eta(t, y_i^{N}(t)) \cdot \fF(y_i^{N}(t), \yy^{N}(t))  \right) \d t \right|^2\\
&\qquad +  \left|  \frac{1}{N} \sum_{i=1}^N \int_0^T \left( \eta(t, x_i^{N,\varepsilon,h}(t))  - \eta(t, y_i^{N}(t)) \right) \cdot \hh^N(t, y_i^{N}(t)) \d t \right|^2 + \left| M^{N,\varepsilon}_h(T)  + \varepsilon\right|^2 \\
&\quad =  I + II + III,
\end{split}
\end{equation}
where we shorthand $M^{N,\varepsilon}_\eta(t) := \frac{\sqrt{\varepsilon}}{N} \sum_{i=1}^N  \int_0^t \eta(t, x_i^{N,\varepsilon,h}(t)) \d W_i(s)$.
For sake of clearness we study the three terms separately.
\begin{equation}
\begin{split}
I &\lesssim  \left| \frac{1}{N} \sum_{i=1}^N\int_0^T \left( \eta(t, x_i^{N,\varepsilon,h}(t))  - \eta(t, y_i^{N}(t)) \right) \cdot \fF(y_i^{N}(t), \yy^{N}(t)) \d t\right|^2 \\
&\quad + \left| \frac{1}{N} \sum_{i=1}^N\int_0^T \eta(t, x_i^{N,\varepsilon,h}(t)) \cdot \left(  \fF(x_i^{N,\varepsilon,h}(t), \xx^{N,\varepsilon,h}(t))  - \fF(y_i^{N}(t), \yy^{N}(t))  \right) \d t \right|^2\\
&\lesssim \left( \sup_{t \in [0,T]} \frac{1}{N} \sum_{i=1}^N\left| x_i^{N,\varepsilon,h}(t) - y_i^{N}(t) \right|\left( 1 + |y_i^N(t)| + \frac{1}{N}\sum_{i=1}^N |y_i^N(t)| \right) \right)^2 \\
&\quad + \left( \sup_{t \in [0,T]} \frac{1}{N} \sum_{i=1}^N \left|  x_i^{N,\varepsilon,h}(t) - y_i^{N}(t) \right| + \sup_{t \in [0,T]} W_1(\mu_t^{N,\varepsilon,h},\sigma_t^N ) \right)^2\\
&\lesssim \left( \sup_{t \in [0,T]} \frac{1}{N} \sum_{i=1}^N\left| x_i^{N,\varepsilon,h}(t) - y_i^{N}(t) \right|^2 \right) \left( 1 + \sup_{t \in [0,T]}\frac{1}{N} \sum_{i=1}^N|y_i^N(t)|^2 \right) \\
&\lesssim \left( \sup_{t \in [0,T]} \frac{1}{N} \sum_{i=1}^N\left| x_i^{N,\varepsilon,h}(t) - y_i^{N}(t) \right|^2 \right), \\
\end{split}
\end{equation}
where we used the inequality $W_1(\mu^{N,\varepsilon,h}_t, \sigma^{N}_t) \leq \frac{1}{N}\sum_{i=1}^N |x_i^{N,\varepsilon,h}(t) - y^N_i(t)|$ and subsequently the uniform control \eqref{eq:control_measures_n} given in Theorem \ref{t:recovery_det}. 
The second part can be estimated in a similar way
\begin{equation}
\begin{split}
II &\lesssim  \left( \frac{1}{N} \sum_{i=1}^N \left( \int_0^T \left| \eta(t, x_i^{N,\varepsilon,h}(t))  - \eta(t, y_i^{N}(t)\right|^2 \d t\right)^{1/2} \left( \int_0^T \left| \hh(t,y_i^{N}(t))\right|^2\d t \right)^{1/2} \right)^2 \\
&\lesssim  \left( \frac{1}{N} \sum_{i=1}^N \int_0^T \left| \eta(t, x_i^{N,\varepsilon,h}(t))  - \eta(t, y_i^{N}(t)\right|^2 \d t\right) \left( \frac{1}{N} \sum_{i=1}^N \int_0^T \left| \hh(t,y_i^{N}(t))\right|^2\d t \right) \\
&\lesssim \left( \sup_{t \in [0,T]} \frac{1}{N} \sum_{i=1}^N\left| x_i^{N,\varepsilon,h}(t) - y_i^{N}(t) \right|^2 \right), \\
\end{split}
\end{equation}
from the uniform bound on $\frac{1}{N} \sum_{i=1}^N \int_0^T \left| \hh(t,y_i^{N}(t))\right|^2\d t$ given by Theorem \ref{t:recovery_det}(e).
Hence, for some constant $C_1 > 0$, 
\begin{equation}\label{conv_I_II}
\P^{N,\varepsilon}_h \left( I + II > \delta \right) \leq \P^{N,\varepsilon}_h \left( \sup_{t \in [0,T]} \frac{C_1}{N} \sum_{i=1}^N\left| x_i^{N,\varepsilon,h}(t) - y_i^{N}(t) \right|^2   > \delta \right) \longrightarrow 0, \quad \text{ as } N\uparrow +\infty, \; \varepsilon \downarrow 0,
\end{equation}
thanks to Proposition \ref{p:wasserstein_esitmate}.
Concerning the last term, there is $C_2 > 0$ such that $[M^{N,\varepsilon}_h]_T \lesssim C_2 \frac{\varepsilon}{N}$, and employing Lemma \ref{l.bernstein} we deduce that
\begin{equation}
\P^{N, \varepsilon}_h \left( \pm M^{N,\varepsilon}_h(T) > \delta \right) \leq e^{- C_2\frac{N}{\varepsilon}\delta} .
\end{equation}
From inequality \eqref{prob_union} we easily get
\begin{equation}\label{conv_III}
\P^{N, \varepsilon}_h \left( |M^{N,\varepsilon}_h(T)|^2 > \delta \right) \leq e^{2- C_2\frac{N}{\varepsilon}\sqrt{\delta}} \longrightarrow 0, \quad \text{ as } N\uparrow +\infty, \; \varepsilon \downarrow 0.
\end{equation}
Collecting the estimates \eqref{conv_I_II} and \eqref{conv_III} we easily get the required result.
\end{proof}

\begin{theorem}\label{t.LDP_lb_1}
Let Hypotheses \ref{h.lip} and \ref{h.comp} be in force. Then for every open set $\sfO \subset \cX$ it holds
\begin{equation}\label{eq:liminf}
\liminf_{\substack{\varepsilon \to 0\\ N \to +\infty}} \frac{\varepsilon}{N} \log \P^{N,\varepsilon} ((\mu^{N}, J^{N,\varepsilon}) \in \sfO) \geq - \inf_{(\mu, J) \in \sfO} I(\mu, J).
\end{equation}
\end{theorem}

\begin{proof}
Assume that $I(\mu,J) < +\infty$, otherwise the result is easily true. 
Lemma \ref{l:rate_h} guarantees that $\d \jJ = \hh \,\d \mu$, for some $\hh \in L^2(Q, \tilde \mu;\R^d)$ and
\begin{equation}
I(\mu,\hh) = \frac{1}{2} \int_0^T \langle \mu_t, |\hh(t,\cdot)|^2 \rangle \d t
\end{equation}
under the constraint $\partial_t \mu_t + \dive \left( (\fF(x,\mu_t) + \hh(x,t)) \mu_t \right) = 0$.

Introduce now the martingale 
\begin{equation}
R^{N,\varepsilon,h}_t := \frac{1}{\sqrt{\varepsilon}} \sum_{i=1}^N \int_0^t \hh^N(t,y_i^N(s)) \d W_i(s),
\end{equation}
where $y_i^N$, $i = 1,\ldots, N$, are solutions to the deterministic equation \eqref{eq:recovery} given in Theorem \ref{t:recovery_det}.
The associated quadratic variation has the form 
\begin{equation}
[R^{N,\varepsilon,h}]_t = \frac{1}{\varepsilon}\sum_{i=1}^N \int_0^t |\hh^N(t,y^N_i(s))|^2 \d s,
\end{equation}
and it is uniformly bounded as $N \uparrow +\infty$ thanks to Theorem \ref{t:recovery_det}-(e).
In order to apply Theorem \ref{t.lb_entropia} we introduce the probability measures 
\[ \bar\P^{N,\varepsilon}_h(\d \omega) = \exp \left( R^{N,\varepsilon,h}_t - \frac{1}{2} [R^{N,\varepsilon,h}]_t \right) (\omega) \P(\d \omega)\]
from which we define $\P^{N,\varepsilon}_h:= \bar \P^{N,\varepsilon}_h \circ (\xx^{N,\varepsilon,h})^{-1}$, where $x_i^{N,\varepsilon,h}$ solves \eqref{eq:tilted}, $i = 1,\ldots, N$. 
Taking advantage from the general inequality \eqref{eq:ent_composiz.} we get   
\begin{equation}
\ent\left( \P^{N,\varepsilon}_h \circ (\mu^N, J^{N,\varepsilon})^{-1} \big| \P^{N,\varepsilon} \circ (\mu^N, J^{N,\varepsilon})^{-1} \right) \leq \ent( \P^{N,\varepsilon}_h | \P^{N,\varepsilon}) \leq \ent(\bar \P^{N,\varepsilon}_h | \P)  
\end{equation}
Then we can compute the rescaled entropy
\begin{equation}
\begin{split}
\frac{\varepsilon}{N} \ent(\bar \P^{N,\varepsilon}_h | \P) &= \frac{\varepsilon}{N} \bar \E^{N,\varepsilon}_h \left( R^{N,\varepsilon,h}_T - \frac{1}{2} [R^{N,\varepsilon,h}]_T \right) \\
&=  \frac{\varepsilon}{N} \bar \E^{N,\varepsilon}_h \left(\frac{1}{2\varepsilon} \sum_{i=1}^N \int_0^T |\hh^N(t,y_i^N(t))|^2 \d t \right) \\
&= \frac{1}{2} \int_0^T \langle \sigma^N_t, |\hh^N(t, \cdot)|^2 \rangle \d t < +\infty,\
\end{split}
\end{equation}
where in the second equality we used Girsanov theorem, stating that $R^{N,\varepsilon,h}_t - [R^{N,\varepsilon,h}]_t$ is a martingale with respect to $\bar \P^{N,\varepsilon}_h$ and it has null expectation.
The last equality comes from the fact the $\sigma_t^N= \sigma[\yy_t^N] = \frac{1}{N}\sum_{i=1}^N \delta_{y_i^{N}(t)}$ is deterministic and uniformly bounded.

Employing now Theorem \ref{t:recovery_det} we finally get       
\begin{equation}
\begin{split}
\limsup_{\substack{\varepsilon \to 0\\ N \to +\infty}} \frac{\varepsilon}{N} \ent(\P^{N,\varepsilon}_h | \P^{N,\varepsilon})  &=  
\limsup_{N \to +\infty} \frac{1}{2} \int_0^T \int_{\R^d} |\hh^N(t,x)|^2 \d \sigma^N_t(x) \d t \\
&\leq \frac{1}{2} \int_0^T \int_{\R^d} |\hh(t, x)|^2 \d \mu_t(x) \d t = I(\mu, h), 
\end{split}
\end{equation}
which is the required bound. 
To conclude the proof it is enough to show that $\P^{N,\varepsilon}_h \circ (\mu^{N,\varepsilon,h}, J^{N,\varepsilon,h})^{-1} \weakto^* \delta_{(\mu,h)}$ as $N \uparrow +\infty, \varepsilon \downarrow 0$. 
With the notation $\delta_{(\mu,h)}$ we mean the probability measure concentrated on the solution $(\mu,\hh)$ to $\partial_t \mu_t + \dive \left( (\fF(x,\mu_t) + \hh(x,t)) \mu_t \right) = 0$.
But this is a consequence of the estimates given in Proposition \ref{p:wasserstein_esitmate} and Proposition \ref{p:diff_current_esitmate}. 
Indeed, take a function $\Psi \in C_b(\cX)$ and define the sets $B^{N,\varepsilon}_\delta \subset C([0,T]; (\R^d)^N)$:
\[B^{N,\varepsilon}_\delta := \left\lbrace  \sup_{t \in [0,T]} W^2_2(\mu^{N,\varepsilon,h}_t, \sigma^{N}_t) > \delta \right\rbrace \cup \left\lbrace \exists \, \eta \in C_c^\infty(Q; \R^d): \left| J^{N,\varepsilon,h}(\eta) - \nu^N(\eta) \right|^2 > \delta \right\rbrace. \]
For every  $l >0$, from the continuity of $\Psi$ there exists $ \bar \delta = \bar \delta (N,\varepsilon) > 0$ such that $\left|\Psi(\mu^{N,\varepsilon, h}, J^{N, \varepsilon, h}) - \Psi(\sigma^N, \nu^N)\right|  \leq l/2$ in $\left( B^{N,\varepsilon}_\delta \right)^c$.  
Then it holds
\begin{equation}
\begin{split}
\E^{N,\varepsilon}_h \left| \Psi(\mu^{N,\varepsilon, h}, J^{N, \varepsilon, h}) - \Psi(\sigma^N, \nu^N) \right| &\lesssim \int_{B_{\bar \delta}^{N,\varepsilon}} \left| \Psi(\mu^{N,\varepsilon, h}, J^{N, \varepsilon, h}) - \Psi(\sigma^N, \nu^N) \right| \d \P^{N,\varepsilon}_h + \frac{l}{2} \\
&\lesssim \P^{N,\varepsilon}_h ( B^{N,\varepsilon}_\delta ) + \frac{l}{2}, 
\end{split}
\end{equation}
where we took advantage of the continuity and boundedness of $\Psi$.
Thanks to Proposition \ref{p:wasserstein_esitmate} and Proposition \ref{p:diff_current_esitmate} there exist $\bar N$, $\bar \varepsilon$ such that $\P^{N,\varepsilon}_h ( B^{N,\varepsilon}_\delta) \leq \frac{l}{2}$, for every $N >\bar N$, $\varepsilon < \bar \varepsilon$, hence
\[\E^{N,\varepsilon}_h \left| \Psi(\mu^{N,\varepsilon, h}, J^{N, \varepsilon, h}) - \Psi(\sigma^N, \nu^N) \right| \lesssim l, \qquad  \forall \, N >\bar N, \, \forall \, \varepsilon < \bar \varepsilon. \]
Employing the continuity of $\Psi$ and the convergence of the sequence $(\sigma^N, \nu^N) \weakto (\mu,\hh \mu)$ in $\cX$ we easily get
\[\E^{N,\varepsilon}_h \Psi(\mu^{N,\varepsilon, h}, J^{N, \varepsilon, h}) \longrightarrow \Psi(\mu, J), \]
which is the required convergence.
\end{proof}

\subsection{Proofs of the main results}

Now we can conclude the proof of the main theorems collecting the results obtained above.

\begin{proof}[Proof of Theorem \ref{t:Limit}]
Thanks to Hypothesis \ref{h.lip} the existence and uniqueness for the system \eqref{eq:state} is fairly standard.
Let us denote by $\sigma^N \in C([0,T]; \PP(\Rd))$ the empirical measure associated to the deterministic system
\[ \frac{\d}{\d t} x_i^{N}(t)=\fF(x_i^{N}(t),\xx^{N}(t)) ,\quad i=1,\cdots, N,\qquad
  \xx^N(0)=\xx^N_0,
 \]
 and by $\nu^N \in \MM([0,T] \times \Rd, \Rd)$ the vector-valued measure 
 \[ \nnu^N(\eta) = \frac{1}{N} \sum_{i=1}^N \int_0^T \eta(t, x_i^{N}(t)) \cdot \fF(x_i^{N}(t), \xx^{N}(t)) \d t, \quad \forall \, \eta \in C_c^\infty (Q; \R^d) \]
Employing Propositions \ref{p:wasserstein_esitmate} and \ref{p:diff_current_esitmate} in the simpler case $\hh^N = 0$ we get
\begin{equation}\label{est1_h_zero}
\begin{split}
\lim_{\substack{\varepsilon \to 0\\ N \to +\infty}} \P^{N, \varepsilon} \left( \sup_{t \in [0,T]} W^2_2(\mu^{N,\varepsilon}_t, \sigma^{N}_t) > \delta \right) &=0, \\
\lim_{\substack{\varepsilon \to 0\\ N \to +\infty}} \P^{N, \varepsilon}_h \left( \left| J^{N,\varepsilon}(\eta) - \nu^N(\eta) \right|^2 > \delta \right) &=0,
\end{split}
\end{equation}
for every $\eta \in C_c^{\infty}(Q; \R^d)$. 
On the other hand, from a compactness argument analogous to the one in the proof of Theorem \ref{ex_uniq_vlasov} (part (i), steps 3-4) there exists $(\mu,\nu) \in \cX$ such that  
\begin{equation}\label{est2_h_zero}
\lim_{N \to +\infty} \sup_{t \in [0,T]} W_1(\sigma^N_t,\mu_t) =0, \qquad \lim_{N \to +\infty} \left| \nu^N(\eta) - \nu(\eta) \right| = 0, \quad \forall\, \eta \in C_c^\infty(Q;\Rd),
\end{equation}
where $\d \nu = F(\cdot, \mu) \d \mu$ and $\partial_t \mu_t + \dive (F(x,\mu_t) \mu_t) = 0$ and the solution is unique, thanks to Theorem \ref{ex_uniq_vlasov} in the Appendix.
The combination of \eqref{est1_h_zero} and \eqref{est2_h_zero} finally guarantees the result.
\end{proof}

\begin{proof}[Proof of Theorem \ref{t:LDP}]
Part (i) of the theorem follows by Lemmata \ref{goodness}  and \ref{l:rate_h}. 
Part (ii) is a consequence of Proposition \ref{p:upper_bound_compacts} and the exponential tightness proved in Theorem \ref{t:exp_tight}.
For what concerns Part (iii) we refer to the proof of Theorem \ref{t.LDP_lb_1}.
\end{proof}

\section{Appendix}

Let us start by presenting a wellposedness result for a class of (nonlocal) Vlasov-type PDEs.
For the sake of clarity, we make a distinction between assumptions for existence and uniqueness of a solution and we furthermore pay attention to the regularity of the initial measure.
 
\begin{theorem}\label{ex_uniq_vlasov}
Let $\vv: [0,T] \times \R^d \times \PP_1(\R^d) \to \R^d$ be a measurable vector field satisfying 
\begin{equation}\label{growth}
\left| \vv(t,x, \mu ) \right| \leq c(t)\left( 1 + |x| + \int_{\R^d}|x| \d \mu(x) \right),
\end{equation}
with $c \in L^1(0,T)$.
Moreover, suppose that $ (t,x) \mapsto \vv(t, x, \mu)$ is a Caratheodory function for every $\mu \in \PP_1(\R^d)$ and
\begin{equation}\label{cont_v}
\sup_{x \in C}\sup_{t \in [0,T]} \left| \vv(t,x,\mu^n) -\vv(t,x,\mu) \right| \longrightarrow 0, \qquad \forall\, C \Subset \Rd,
\end{equation}
whenever $W_1(\mu^n, \mu) \to 0$ as $n \uparrow +\infty$.
The following results hold true
\begin{itemize}
\item[(i)] If $\mu_0 \in \PP(\Rd)$ and $\int_{\Rd} \psi(x) \d \mu_0(x) < +\infty$ for a moderated function $\psi: \R^d \to \R_+$ (in the sense of \cite[Def.~2.2]{fornasier2018mean}), then there exists a distributional solution $\mu \in C([0,T]; \PP_1(\Rd))$ to the  equation
\begin{equation}\label{CE}
\partial_t \mu_t + \dive (\vv(t,x,\mu_t) \mu_t) = 0
\end{equation}
with initial datum $\mu(0,x) = \mu_0(x)$.
\item[(ii)] When the map  $(x,\mu) \mapsto \vv(\cdot, x, \mu)$ is Lipschitz:
\begin{equation}\label{lip_v}
\left| \vv(t,x,\mu) - \vv(t,y,\nu) \right| \leq L\left( |x-y|  + W_1(\mu, \nu)\right), \qquad \forall \, x,y \in \Rd, \;\mu,\nu \in \PP_1(\Rd),
\end{equation}
then equation \eqref{CE} admits a unique solution.
\item[(iii)] If $\supp(\mu_0) \subset \cC$, for some compact set $\cC \Subset \R^d$, then there exists $R>0$ such that $\supp(\mu_t) \subset B(0,R)$, for every $t \in [0,T]$.
\end{itemize} 
\end{theorem}

\begin{proof}
In the proof assume for simplicity that the function $c \in L^1(0,T)$ appearing in \eqref{growth} is actually constant.
For the more general case the technique is equivalent.

\smallskip
\noindent (i) Let us divide the proof in several steps.

STEP 1: (A priori estimate) 
Suppose here that there exists a solution to equation \eqref{CE}.
Thanks to \cite[Prop.~5.3]{fornasier2018mean} we can find an admissible function $\theta: [0,+\infty) \to [0,+\infty)$ in the sense of \cite[Def.~2.2]{fornasier2018mean} such that 
\begin{equation}\label{apriori_theta}
\sup_{t \in [0,T]} \int_{\Rd} \theta(|x|) \d \mu_t(x) \leq C \left( 1 + \int_{\Rd} \psi(x) \d\mu_0(x) \right),
\end{equation}
where the constant $C$ depends only on $c, T, \int_\Rd |x| \d \mu_0(x)$ and the doubling constant $K$ of $\theta$.

STEP 2: (Approximation of initial data) 
Starting from $\mu_0 \in \PP(\Rd)$ we define a sequence of compact sets $A_k$ with the property 
\[A_k \subset A_{k+1}, \quad \mu_0(A_k) < \dfrac{1}{k}, \quad  \lim_{k \to +\infty} \mu_0(A_k) = \mu_0\left( \bigcup_{i=1}^\infty A_i \right) = 1. \]
We introduce the normalized measures $\mu_0^k \in \PP(\R^d)$
\[ \mu_0^k:= \frac{1}{\mu_0(A_k)}\mu_0 \mres{A_k}, \qquad \mu_0^k \weakto \mu_0, \quad \text{ as } k \to +\infty, \]
and using Beppo Levi monotone convergence we get
\begin{equation}\label{conv_1}
 \lim_{k \to +\infty} \left| \int_{\Rd}\psi(x)\d \mu_0^k(x)  - \int_{\Rd}\psi(x)\d \mu_0(x) \right| = 0.
\end{equation}
Thanks to the compactness of $A_k$ there exists a sequence of empirical measures $\mu_0^{k,m}$ such that
\[ \mu_0^{k,m}:= \frac{1}{m}\sum_{j=1}^m \delta_{x_{j,k}}, \qquad \mu_0^{k,m} \weakto \mu_0^k, \quad  \text{ as } m \to +\infty. \] 
Moreover, $\psi$ is bounded and continuous on $A_k$ so that
\[ \lim_{m \to +\infty} \left| \int_{\Rd}\psi(x)\d \mu_0^{k,m}(x)  - \int_{\Rd}\psi(x)\d \mu^k_0(x) \right| = 0.\]
This implies that for every $k \in \N$ there exists $\bar m(k)$ for which
\begin{equation}\label{conv_2}
\left| \int_{\Rd}\psi(x)\d \mu_0^{k,m}(x)  - \int_{\Rd}\psi(x)\d \mu^k_0(x) \right| \leq \frac{1}{k}, 
\end{equation}
and if we define $\bar\mu^k_0 := \mu^{k,\bar m(k)}_0$ it holds that
\begin{equation}\label{conv_3}
 \lim_{k \to +\infty} \left| \int_{\Rd}\psi(x)\d \bar\mu_0^k(x)  - \int_{\Rd}\psi(x)\d \mu_0(x) \right| = 0,
\end{equation} 
where we used \eqref{conv_1} and \eqref{conv_2}.
This easily implies that $W_1(\bar \mu^k_0, \mu_0) \to 0$, when $k \uparrow +\infty$, thanks to the superlinearity of $\psi$.
Moreover 
\begin{equation}\label{unif_control}
\sup_{k \in \N} \int_{\Rd}\psi(x) \d \bar\mu^k_0(x) <+\infty.
\end{equation}

STEP 3: (Compactness) Given $k \in \N$, from the previous step we get a finite set of initial data $x_{0,1}^k, \ldots, x_{0,k}^k \in \R^d$, hence we can introduce the system of characteristics
\begin{equation}\label{system}
\begin{system}
\dot x_i^k(t) = \vv(t,x_i^k(t),\xx(t)) \qquad t \in [0,T] \\
x^k_i(0) = x_{0,i}^k.
\end{system}
\end{equation}
Existence of a solution to \eqref{system} is guaranteed by the regularity of $\vv$. Indeed, defining $\vV(t,\xx) := ( \vv(t,x_1,\xx);$ $\ldots; \vv(t,x_k,\xx) )$, the system can be written as $\dot\xx(t) = \vV(t, \xx(t))$, where $\vV$ is Charateodory thanks to the continuity of $\vv$ w.r.t $x$ and condition \eqref{cont_v}.	
Now we can associate to $\xx(t)$ the empirical measure 
\begin{equation}\label{empirical_k}
\mu^k_t:= \frac{1}{k} \sum_{i=1}^k \delta_{x_i^k(t)}, \quad \forall\, t \in[0,T],
\end{equation}
which actually solves equation \eqref{CE} with initial condition $1/k \sum_{i=1}^k \delta_{x_{0,i}^k(t)}$.
It is easy to show (see e.g. \cite[Lem.~4.1]{fornasier2018mean}) that 
\begin{equation}
\sup_{t \in [0,T]} \int_{\Rd} |x| \d \mu^k_t(x) \leq \tilde C \left( 1 + \int_{\Rd} |x|\d \bar \mu_0^k(x) \right),
\end{equation}
where the constant $\tilde C$ depends only on $C, T$.
Thanks to the apriori estimate given in step 1 and the uniform control on the initial data in \eqref{unif_control} we actually get a stronger estimate
\begin{equation}\label{theta_estimate}
\sup_{k \in \N} \sup_{t \in [0,T]} \int_\Rd \theta(|x|) \d \mu_t^k(x) <+\infty.
\end{equation}
Furthermore, if $s \leq t \in [0,T]$ we compute
\begin{equation}\label{equicontinuity_k}
\begin{split}
W_1(\mu^k_s, \mu^k_t) &\leq \frac{1}{k} \sum_{i=1}^k \left| x_i^k(t) - x_i^k(t) \right| \leq \frac{1}{k} \sum_{i=1}^k  \int_s^t \left|  	 \vv(r,x_i^k(r), \mu^k_r) \right| \d r \\
&\leq |t-s| \int_\Rd |x| \d \mu_t^k(x) \leq l |t-s|, 
\end{split}
\end{equation} 
for some positive constant $l \in \R_+$ independent of $k$.
From the application of Ascoli-Arzel\`a theorem we get the existence of a limit curve $\mu \in C([0,T]; \PP_1(\Rd))$ such that
\begin{equation}\label{conv_wass_muk}
\lim_{k \to +\infty} \sup_{t \in [0,T]} W_1(\mu^k_t, \mu_t) = 0
\end{equation}

STEP 4: (Identification of the limit) To show that the candidate limit is a solution to equation \eqref{CE} let us check that for every $\varepsilon >0 $ there exists $\bar k$ such that for every $k \geq \bar k$
\begin{equation}\label{limit_eq_conv}
\left| \int_0^t\int_\Rd \nabla \phi(x) \cdot \vv(t,x,\mu^k_t) \d \mu_t^k(x)\d t - \int_0^t\int_\Rd \nabla \phi(x) \cdot \vv(t,x,\mu_t) \d \mu_t(x)\d t \right| < \varepsilon, \qquad \forall\, \phi \in C_c^1(Q; \Rd). 
\end{equation}
Thanks to \eqref{conv_wass_muk} there exists a compact set $\cC \subset \Rd$ (see e.g. \cite[pag.~22]{fornasier2018mean}) such that 
\[ \sup_{t \in [0,T]}\int_{\Rd \setminus \scC} \left( 1 + |x| \right) \d(\mu_t + \mu_t^k) \leq \varepsilon. \]
Hence \eqref{growth} assures that 
\begin{equation}
\left| \int_0^t\int_{\Rd \setminus \scC} \nabla \phi(x) \cdot \vv(t,x,\mu^k_t) \d \mu_t^k(x)\d t - \int_0^t\int_{\Rd \setminus \scC} \nabla \phi(x) \cdot \vv(t,x,\mu_t) \d \mu_t(x)\d t \right| < 2\varepsilon. 
\end{equation}
It remains to estimate
\begin{equation}
\begin{split}
&\left| \int_0^t\int_{\scC} \nabla \phi(x) \cdot \vv(t,x,\mu^k_t) \d \mu_t^k(x)\d t - \int_0^t\int_{\scC} \nabla \phi(x) \cdot \vv(t,x,\mu_t) \d \mu_t(x)\d t \right| \\
&\quad \leq  \left| \int_0^t\int_{\scC} \nabla \phi(x) \cdot \left( \vv(t,x,\mu^k_t) - \vv(t,x,\mu_t) \right)\d \mu^k_t(x)\d t \right| \\
&\quad +  \left| \int_0^t\int_{\scC} \nabla \phi(x) \cdot \vv(t,x,\mu_t)\left( \d \mu^k_t(x) - \d \mu_t(x) \right)\d t \right| \\
&\quad \leq \sup_{x\in \scC} \sup_{t \in [0,T]} |\vv(t,x,\mu_t^k) - \vv(t,x,\mu_t)| + \int_0^t\int_{\scC} \left( 1 + |x| \right)\left( \d \mu^k_t(x) - \d \mu_t(x) \right) \d t. \\
\end{split}
\end{equation}
Using \eqref{cont_v} and \eqref{conv_wass_muk} we can find $\bar k$ such that for every $k \geq \bar k$ the above quantity is controlled by $\varepsilon$. This conclude the proof, as the other terms of the (weak formulation of the) equation easily pass to the limit.

\smallskip
(ii) In the following, given a Borel vector field $\ww_t$ and the system of characteristics $\dot x_t = \ww(t,x_t)$, we will write 
$\cT_t: \R^d \to \R^d$ for the associated flow: $\cT_t(x_0) = x_t$, for every $t \in[0,T]$. 

Let us now assume that condition \eqref{lip_v} is in force and fix $\nu \in C([0,T];\PP_1(\R^d))$.
Then we have
\[|\vv(t,x,\nu)| \leq c(1 + |x|) \qquad \text{ and } \quad |\vv(t,x,\nu_t) - \vv(t,y,\nu_t)| \leq L|x-y|, \quad \forall \, x,y \in \R^d, \]
and from the general theory (see e.g. \cite[Prop.~8.1.8]{ambrosio2008gradient}) the system of characteristics $\dot x_t = \vv(t,x_t,\nu_t)$ admits a unique solution $\cT_t^\nu(x_0) = x_t$ which is globally defined in $[0,T]$ for $\mu_0$-a.e. $x_0 \in \R^d$.
Moreover, $\mu_t = (\cT_t^\nu)_\sharp \mu_0$ for every $t \in [0,T]$ is the unique solution to $\partial_t \mu_t + \nabla \cdot (\vv(t,x,\nu_t) \mu_t) = 0$.
Let us show that the map $\nu \mapsto \cT^\nu$ is a strict contraction in the space $C([0,T];\PP_1(\R^d))$.

From the a priori estimate \eqref{apriori_theta} we already know that the map $\cT: C([0,T];\PP_1(\R^d)) \to C([0,T];\PP_1(\R^d))$ is well defined.
Given $\nu^1, \nu^2 \in C([0,T];\PP_1(\R^d))$ we can compute
\begin{equation}
\begin{split}
\left| \cT_t^{\nu^1}(y) - \cT_t^{\nu^2}(y)  \right| &= \left| x^{\nu^1}_t - x^{\nu^2}_t  \right| \leq \int_0^t |\vv(s, x^{\nu^1}_s, \nu^1_s) - \vv(s, x^{\nu^2}_s, \nu^2_s)| \d s \\
&\leq L \int_0^t \left( |x^{\nu^1}_s - x^{\nu^2}_s| + W_1(\nu^1_s, \nu^2_s) \right) \d s \\
&\leq  Lte^t \sup_{s \in [0,T]} W_1(\nu^1_s, \nu^2_s)
\end{split}
\end{equation}
where  we used Gronwall inequality in the form \eqref{gronwall}. 
Hence
\begin{equation}
\begin{split}
W_1\left((\cT_t^{\nu^1})_\sharp (\mu_0), (\cT_t^{\nu^2})_\sharp (\mu_0)\right) &\leq \int_{\R^d} \left| \cT_t^{\nu^1}(y) - \cT_t^{\nu^2}(y)  \right| \d \mu_0(y) \\
&\leq C(t) \sup_{s \in [0,T]} W_1(\nu^1_s, \nu^2_s)
\end{split}
\end{equation}
where $C(t) \to 0$ as $t \downarrow 0$. If we choose $\tilde T$ so that $C(\tilde T) < 1$ we get 
\[ \sup_{t \in [0,\tilde T]} W_1\left((\cT_t^{\nu^1})_\sharp (\mu_0), (\cT_t^{\nu^2})_\sharp (\mu_0)\right) < \sup_{s \in [0,T]} W_1(\nu^1_s, \nu^2_s),\]
which assures the existence of a unique fixed point of the map $\cT$. 
Using global existence in time of the solution the argument can be easily extended to prove uniqueness on the whole interval $[0,T]$.

\smallskip
(iii) When $\supp (\mu_0) \subset \cC$, we firstly select $x^k_{0,1}, \ldots , x^k_{0,k} \in \cC$ and $\mu_0^k:= \mu^k[\xx_0]$ with the property that $W_1(\mu_0, \mu_0^k) \to 0$ as $k \uparrow +\infty$.
Let us now estimate the growth of $|x_i^k(t)|$, $i = 1, \ldots, k$, by writing 
\begin{equation}
\begin{split}
|x_i^k(t)| &\leq |x_{0,i}^k| + \int_0^t |\vv(s,x_i^k(s),\xx^k)| \d s \\
&\leq \max_{i=1, \ldots, k} |x_{0,i}^k| + c\int_0^t \left( 1 + 2\max_{i=1, \ldots, k} |x^k_i(s)| \right) \d s.
\end{split}
\end{equation}
Denoting by $\bar R:= \max \lbrace |x|: x \in \cC  \rbrace $ and using Gronwall lemma we get that
\begin{equation}
\sup_{t \in [0,T]}\max_{i=1, \ldots, k} |x^k_i(t)| \leq R,
\end{equation}
where $R:= (\bar R + cT)e^{2T}$.
Then if we  define $\mu^k_t := \mu^k[\xx_t]$ the sequence $\mu^k_t  \in C([0,T];\PP_1(B(0,R)))$ is uniformly supported in $B(0,R)$.
Using a compactness argument as in part (i) of the proof, we easily get that the limit measures also satisfy $\sup_{t \in [0,T]} \supp(\mu_t) \subset B(0,R)$ and we conclude.
\end{proof}

Let us finally show the following 
\begin{proof}[Proof of Theorem \ref{t:reg_stoch_current}]
The proof follows the same lines of \cite[Thm.~C.1]{bertini2017stochastic} (see also \cite{flandoli2005stochastic}).
We sketch here the main steps of the proof for the convenience of the reader. 

\medskip

\textit{Step 1}. 
There exists $\mathcal{C} \in L^2(\Omega; \R_+)$ such that $| J(\eta) | \leq \mathcal{C}(\omega) \|  \eta\|_{\shH^s}$.
To prove it, let us write
\begin{equation}
J(\eta) = \int_0^T \eta(t,X_t) \d X_t + \frac{1}{2}\sum_{i,j=1}^d \int_0^T \frac{\partial \eta_i}{\partial x_j}(X_t) \d [M_i,M_j]_t.
\end{equation}
and use Fourier inversion formula along with stochastic Fubini theorem.
Precisely, define the functions $e^m_{n,k}: Q \to \C^d$:
\begin{equation}
e^m_{n,k}(t,x):= \sqrt{\frac{2-\delta_{n,0}}{T}} \cos\left(\frac{n\pi t}{T}\right) \frac{e^{i  k \cdot x}}{(2\pi)^{d/2}}e_m,
\end{equation}
where $e_1 \ldots, e_d$ is the canonical basis in $\R^d$ and $n \in \Z_+$, $k \in \R^d$. 
Given $\eta$ as above we denote by $\hat \eta_n^m(k)$ its Fourier coefficients:
\begin{equation}
\hat \eta_n^m(k):= \int_0^T \int_{\R^d} [e^m_{n,k}]^* \cdot \eta(t,x) \d x \d t,
\end{equation}
where $*$ denotes the complex conjugation.
Hence, using \cite[Lemma~8]{flandoli2005stochastic} we get that $\P$-a.s.
\begin{equation}\label{eq:J_eta}
J(\eta) = \sum_{m,n}\int_{\R^d} \hat \eta_n^m(k)^* Z^m_n(k) \d k,
\end{equation}
where we set 
\begin{equation}\label{eq:def_Z^n}
\begin{split}
Z^m_n(k) &:= \int_0^T e^m_{n,k}(t,X_t) \circ \d X_t \\
&= \int_0^T e^m_{n,k}(t,X_t) \d X_t + \frac{i}{2} \sum_{j=1}^d  k_j \int_0^T  e^m_{n,k}(t,X_t) \d [X^j,X^m]_t.
\end{split}
\end{equation}
The above coefficients can be controlled in $L^2(\Omega)$ with a constant $C = C(T,d)$:
\begin{equation}
\begin{split}
\E \left( |Z^m_n(k)|^2 \right) &= 2\E \left( \left| \int_0^T e^m_{n,k}(t,X_t) \d M_t \right|^2 \right) + 2\E \left( \left| \int_0^T e^m_{n,k}(t,X_t) \d V_t \right|^2 \right) \\
&+ C|k|^2 \sum_{j=1}^d \E \left( \left| \int_0^T e^m_{n,k}(t,X_t) \d [X^j,X^m]_t \right|^2 \right) \\
&\leq 2 \| e^m_{n,k} \|_{\infty}^2 \left( \E[M]_T + \E|V|_T \right) + C|k|^2 \| e^m_{n,k} \|_{\infty}^2  \sum_{j} \left( [X^j]_T + [X^m]_T  \right) \\
&\leq C(1 + |k|^2),
\end{split}
\end{equation} 
where we used the fact that $[X^j,X^m]_t $ has bounded variation and can be written as $\frac{1}{4}[X^j  + X^m]_t - \frac{1}{4}[X^j  - X^m]_t$. 
Moreover we supposed that $\E[M]_T + \E|V|_T + [X^j]_T \leq C$: a rigorous localization procedure to justify this bound can be found in \cite[Thm.~9]{flandoli2005stochastic}.
Now if we extend $\eta$ to an even  function on $[-T,T]$, an equivalent norm in $\hH^s$ is given by
\begin{equation}
\|\eta\|^2_{\shH^s} = \sum_{n,m} \int_{\R^d} (1+n^2)^{s_1}(1+|k|^2)^{s_2} |\hat \eta_n^m(k)|^2 \d k
\end{equation}
Hence, applying Cauchy-Schwartz inequality in \eqref{eq:J_eta} we get 
\begin{equation}\label{J_est}
|J(\eta)|^2 \leq \mathcal{\tilde C} \| \eta \|^2_{\shH^s},
\end{equation}
where 
\begin{equation}\label{J_norm}
\mathcal{\tilde C} := \sum_{n,m} \int_{\R^d} (1+n^2)^{-s_1}(1+|k|^2)^{-s_2} |Z_n^m(k)|^2 \d k
\end{equation}
and 
\begin{equation}
\E \mathcal{\tilde C} \leq C \sum_{n} \int_{\R^d} (1+n^2)^{-s_1}(1+|k|^2)^{-s_2+1} \d k < \infty,
\end{equation}
due to the choice of $(s_1, s_2) \in (\frac{1}{2},1) \times (\frac{d+2}{2}, +\infty)$. So that $| J(\eta) | \leq \mathcal{C}(\omega) \|  \eta\|_{\shH^s}$. \\

\medskip

\textit{Step 2}. We show how the previous step guarantees the existence of the map $\cJ$.\\
Let $\Phi: \Omega \to \hH^s$ a simple function defined as $\Phi = \sum_{i} \eta_i 1_{\Omega_i}$, where $\eta_i \in \hH^s$ and $(\Omega_i)$ a finite partition of $\Omega$. 
Consider the functional
\begin{equation}
\cB(\Phi) = \E \left( \sum_i J(\eta_i) 1_{\Omega_i} \right).
\end{equation} 
From the previous step and Cauchy-Schwartz inequality we have
\begin{equation}
|\cB(\Phi)| \leq \E \left(\sum_i \mathcal{C} \|  \eta_i\|_{\shH^s} 1_{\Omega_i} \right) \leq \| \mathcal{C} \|_{L^2(\Omega)}\| \Phi \|_{L^2(\Omega; \shH^s)},
\end{equation} 
so that  $\cB$ is a linear and bounded operator, whence extendible by density  to $L^2(\Omega; \hH^s)$.
By Riesz representation theorem there exists $\Psi \in L^2(\Omega; \hH^{-s})$ such that 
\begin{equation}
\E \left( \langle \Psi, \eta \rangle 1_{\Omega'} \right) = \cB(\eta 1_{\Omega'}), \qquad  \forall \, \Omega' \subset \Omega, \eta \in \hH^s.
\end{equation}
From the arbitrariness of $\Omega'$ we get that $\langle \Psi, \eta \rangle = J(\eta)$ $\P$-a.s., and choosing any representative $\cJ: \Omega \to \hH^s$ of $\Psi$ we get the required result.
\end{proof}

\subsection*{Acknowledgements} 
I am grateful to Lorenzo Bertini for introducing me to the problem and providing invaluable help. I also thank Mauro Mariani for enlightening discussions on some aspects of this work.

\bibliography{biblio}
\bibliographystyle{plain}

\end{document}

%% file: macro2014.tex
\newcommand{\C}{\mathbb{C}}

\newcommand{\E}{\mathbb{E}}

\newcommand{\N}{\mathbb{N}}

\renewcommand{\P}{\mathbb{P}}

\newcommand{\R}{\mathbb{R}}

\newcommand{\Z}{\mathbb{Z}}

\newcommand{\DD}{\mathscr{D}}

\newcommand{\MM}{\mathscr{M}}

\newcommand{\PP}{\mathscr{P}}

\newcommand{\cA}{{\ensuremath{\mathcal A}}}
\newcommand{\cB}{{\ensuremath{\mathcal B}}}

\newcommand{\cF}{{\ensuremath{\mathcal F}}}
\newcommand{\cG}{{\ensuremath{\mathcal G}}}

\newcommand{\cJ}{{\ensuremath{\mathcal J}}}
\newcommand{\cK}{{\ensuremath{\mathcal K}}}

\newcommand{\cM}{{\ensuremath{\mathcal M}}}

\newcommand{\cP}{{\ensuremath{\mathcal P}}}
\newcommand{\cQ}{{\ensuremath{\mathcal Q}}}

\newcommand{\cT}{{\ensuremath{\mathcal T}}}

\newcommand{\cX}{{\ensuremath{\mathcal X}}}

\newcommand{\cW}{{\ensuremath{\mathcal W}}}

\renewcommand{\gg}{{\mbox{\boldmath$g$}}}
\newcommand{\hh}{{\mbox{\boldmath$h$}}}

\newcommand{\vv}{{\mbox{\boldmath$v$}}}
\newcommand{\ww}{{\mbox{\boldmath$w$}}}
\newcommand{\xx}{{\mbox{\boldmath$x$}}}
\newcommand{\yy}{{\mbox{\boldmath$y$}}}

\newcommand{\cC}{{\mbox{\boldmath$C$}}}

\newcommand{\fF}{{\mbox{\boldmath$F$}}}
\newcommand{\gG}{{\mbox{\boldmath$G$}}}
\newcommand{\hH}{{\mbox{\boldmath$H$}}}

\newcommand{\jJ}{{\mbox{\boldmath$J$}}}

\newcommand{\vV}{{\mbox{\boldmath$V$}}}

\newcommand{\scC}{{\mbox{\scriptsize\boldmath$C$}}}

\newcommand{\shH}{{\mbox{\scriptsize\boldmath$H$}}}

\newcommand{\nnu}{{\mbox{\boldmath$\nu$}}}

\newcommand{\sfB}{{\sf B}}
\newcommand{\sfC}{{\sf C}}

\newcommand{\sfK}{{\sf K}}

\newcommand{\sfO}{{\sf O}}

\newcommand{\Kliminf}{K\kern-3pt-\kern-2pt\mathop{\rm
lim\,inf}\limits}  
\newcommand{\Klimsup}{K\kern-3pt-\kern-2pt\mathop{\rm lim\,sup}\limits}  
\newcommand{\supp}{\mathop{\rm supp}\nolimits}   

\newcommand{\tr}{\mathop{\rm tr}\nolimits}     
\renewcommand{\d}{{\mathrm d}}

\newcommand{\restr}[1]{\lower3pt\hbox{$|_{#1}$}}

\newcommand{\weakto}{\rightharpoonup}

\newcommand{\nchi}{{\raise.3ex\hbox{$\chi$}}}
\newcommand{\Rd}{{\R^d}}

\def\qed{\ifmmode 
  \else \leavevmode\unskip\penalty9999 \hbox{}\nobreak\hfill
  \fi               
    \qquad           \hbox{\hskip.5em $\square$
                \hskip.1em}}
\def\endproofsym{\qed}

\newenvironment{proof}[1][Proof]{\def\endproofsym{\qed}\trivlist\item[\hskip\labelsep{%
\noindent{\normalfont\emph{#1}.}\hskip .321429\parindent}]\ignorespaces}
{\endproofsym\endtrivlist}

